\documentclass[preprint,a4paper,3p]{elsarticle}
\usepackage{soul}
\usepackage[utf8]{inputenc}

\usepackage{xcolor}
\definecolor{purple}{rgb}{0.65, 0.27, 0.95}
\definecolor{rose}{rgb}{0.8, 0, 0.4}
\definecolor{vert}{rgb}{0, 0.6, 0.4}

\usepackage{amsmath}
\usepackage{amsthm}
\usepackage{xspace}
\usepackage{amsfonts}
\usepackage{comment}
\usepackage{mathtools}
\usepackage{mathrsfs}

\makeatletter
\DeclareRobustCommand\widecheck[1]{{\mathpalette\@widecheck{#1}}}
\def\@widecheck#1#2{%
    \setbox\z@\hbox{\m@th$#1#2$}%
    \setbox\tw@\hbox{\m@th$#1%
       \widehat{%
          \vrule\@width\z@\@height\ht\z@
          \vrule\@height\z@\@width\wd\z@}$}%
    \dp\tw@-\ht\z@
    \@tempdima\ht\z@ \advance\@tempdima2\ht\tw@ \divide\@tempdima\thr@@
    \setbox\tw@\hbox{%
       \raise\@tempdima\hbox{\scalebox{1}[-1]{\lower\@tempdima\box
\tw@}}}%
    {\ooalign{\box\tw@ \cr \box\z@}}}
\makeatother

\usepackage{tikz}
\tikzstyle{vertex}=[circle, draw]
\usetikzlibrary{shapes.multipart}
\usetikzlibrary{shapes,arrows,oodgraph,positioning,decorations.pathreplacing}

\def\smallskip{\vspace\smallskipamount}
\def\medskip{\vspace\medskipamount}
\def\bigskip{\vspace\bigskipamount}
\newskip\smallskipamount \smallskipamount=3pt plus 1pt minus 1pt
\newskip\medskipamount   \medskipamount  =6pt plus 2pt minus 2pt
\newskip\bigskipamount   \bigskipamount =12pt plus 4pt minus 4pt

\newcommand{\wrt}{\emph{w.r.t.}\@\xspace}
\newcommand{\ie}{\emph{i.e.}\@\xspace}

\usepackage[ruled,vlined,linesnumbered]{algorithm2e}
\def\newv#1{#1}

\DeclareMathOperator{\lcm}{lcm}

\usepackage{stackengine}
\stackMath
\usepackage{tabularx}
\usepackage{nicematrix}

\usetikzlibrary{calc}

\usepackage{url}

\tikzset{
    ncbar angle/.initial=90,
    ncbar/.style={
        to path=(\tikztostart)
        -- ($(\tikztostart)!#1!\pgfkeysvalueof{/tikz/ncbar angle}:(\tikztotarget)$)
        -- ($(\tikztotarget)!($(\tikztostart)!#1!\pgfkeysvalueof{/tikz/ncbar angle}:(\tikztotarget)$)!\pgfkeysvalueof{/tikz/ncbar angle}:(\tikztostart)$)
        -- (\tikztotarget)
    },
    ncbar/.default=0.5cm,
}

\tikzset{square left brace/.style={ncbar=0.5cm}}
\tikzset{square right brace/.style={ncbar=-0.5cm}}

\tikzset{round left paren/.style={ncbar=0.5cm,out=120,in=-120}}
\tikzset{round right paren/.style={ncbar=0.5cm,out=60,in=-60}}

\makeatletter
\newtheorem{repeatthm@}{Theorem}

\makeatother

\makeatletter
\newtheorem{repeatlemma@}{Lemma}

\makeatother

\theoremstyle{definition}
\newtheorem{definition}{Definition}[section]
\newtheorem{proposition}[definition]{Proposition}
\newtheorem{theorem}[definition]{Theorem}
\newtheorem{lemma}[definition]{Lemma}
\newtheorem{conjecture}[definition]{Conjecture}
\newtheorem{corollary}[definition]{Corollary}

\usepackage{blkarray}

\newcommand{\setst}{V}

\newcommand{\structure}[1]{\ensuremath{\left(#1\right)}\xspace}

\newcommand{\nodedds}{v}

\newcommand{\set}[1]{\ensuremath{\left\{#1\right\}}\xspace}

\newcommand{\dds}{G}
\newcommand{\ddsps}{H}
\newcommand{\ddspps}{K}

\newcommand{\tra}{h}

\newcommand{\ori}{o}

\newcommand{\oris}{\mathcal{O}}

\newcommand{\uheight}{\mathcal{H}}

\newcommand{\multi}{T}
\newcommand{\matrixt}[2]{\multi^{#1}_{#2}}

\newcommand{\ind}{g}
\newcommand{\indexp}{r}

\newcommand{\pp}{\mathcal{P}}

\newcommand{\tp}{\mathcal{T}}
\newcommand{\tpmatrix}[2]{N^{#1}_{#2}}

\newcommand{\ncmp}{\ell}

\newcommand{\N}{\mathbb{N}}

\newcommand{\ring}{\mathcal{D}}

\newcommand{\coef}{A}

\newcommand{\rterm}{B}

\newcommand{\qtvar}{\nu}

\newcommand{\mon}{m}

\newcommand{\esp}{w}

\newcommand{\cycle}{\mathcal{C}}

\newcommand{\tabs}[1]{\check{#1}}

\newcommand{\multiset}[1]{[#1]}

\newcommand{\unroll}[2]{U_{#2}(#1)}

\newcommand{\roll}[2]{R_{#2}(#1)}

\newcommand{\vunroll}{\mathcal{V}}
\newcommand{\eunroll}{\mathcal{E}}

\newcommand{\unrollp}{\star}

\newcommand{\iso}{\cong}
\newcommand{\notiso}{\not\cong}

\newcommand{\cut}[2]{C_{#2}(#1)}

\newcommand{\homo}[1]{\mathrm{hom}(#1)}
\newcommand{\homset}[1]{\mathrm{Hom}(#1)}

\newcommand{\itree}{I}

\newcommand{\bijHom}{\varphi}

\newcommand{\ftree}{F}

\newcommand{\onehom}{\tau}

\usepackage{algorithm2e}
\SetKwComment{Comment}{$\triangleright$\ }{}

\begin{document}

\title{Decomposition and factorisation of transients in Functional Graphs}

\author[UCA]{François Doré}
\ead{francois.dore@univ-cotedazur.fr}
\author[UCA]{Enrico Formenti}
\ead{enrico.formenti@univ-cotedazur.fr}
\author[UP]{Antonio E. Porreca}
\ead{antonio.porreca@lis-lab.fr}
\author[UCA,JSP]{Sara Riva}
\ead{sara.riva@univ-cotedazur.fr}

\address[UCA]{Universit\'e C\^ote d'Azur, CNRS, I3S, France}
\address[JSP]{Univ. Bordeaux, CNRS, Bordeaux INP, LaBRI, UMR 5800, F-33400 Talence, France}
\address[UP]{Université Publique, France}
\begin{abstract} 
Functional graphs (FGs) model the graph structures used to analyse the behaviour of functions from a discrete set to itself. In turn, such functions are used to study real complex phenomena evolving in time. As the systems involved can be quite large, it is interesting to decompose and factorise them into several subgraphs acting together. Polynomial equations over functional graphs provide a formal way to represent this decomposition and factorisation mechanism, and solving them validates or invalidates hypotheses on their decomposability. The current solution method breaks down a single equation into a series of \emph{basic} equations of the form $A\times X=B$ (with $A$, $X$, and $B$ being FGs) to identify the possible solutions. However, it is able to consider just FGs made of cycles only. This work proposes an algorithm for solving these basic equations for general connected FGs. 
By exploiting a connection with the cancellation problem, we prove that the upper bound to the number of solutions is closely related to the size of the cycle in the coefficient $A$ of the equation. The cancellation problem is also involved in the main algorithms provided by the paper. We introduce a polynomial-time semi-decision algorithm able to provide constraints that a potential solution will have to satisfy if it exists. Then, exploiting the ideas introduced in the first algorithm, we introduce a second exponential-time algorithm capable of finding all solutions by integrating several `hacks' that try to keep the exponential as tight as possible. 
\end{abstract}
\begin{keyword}
Functional Graph, Transient Dynamics, Graph Direct Product
\end{keyword}

\maketitle

\section{Introduction}
\label{sec:intro}

Functional graphs (FGs) are directed graphs with outdegree 1.
They are structurally equivalent to well-known formal models called Discrete-time (discrete-space) Dynamical Systems (DDSs). If the set of nodes is finite, then each node is \emph{ultimately periodic}, \ie, a path originating in a state will eventually run into a cycle. Hence, these graphs are characterised by a finite number of components (\ie, connected components of the underlying undirected graph) with just one cycle each. Then, the nodes can be classified as \emph{cyclic} if they belong to a cycle, or \emph{transient} otherwise.

From a mathematical point of view, FGs are a simple model, but they find important applications in the description of the dynamics of many well-known discrete systems, namely, Boolean automata and  networks, genetic regulatory networks, cellular automata, and many others \cite{Sene12, bower2004computational, AlonsoSanz12}. Over time, an important research direction has been to investigate the expected number of components of a random mapping. This question was opened by Metropolis and Ulam in 1953 \cite{metropolis1952property} and answered by Kruskal one year later \cite{kruskal1954expected}. Since then, several research directions have emerged to study, for example, the number of cycles, trajectories, and the number and the size of the components of randomly generated FGs \cite{romero2005grasping}. Considering FGs, the components are a decomposition of the set of nodes in disjoint minimal non-empty invariant sets (\ie, $f^{-1}(\setst)=\setst$). Given this important aspect, Katz studied the probability that a random mapping is indecomposable (\ie, the graph has just one component) \cite{katz1955probability}, while Romero and Zertuche introduced explicit formulae for the statistical distribution of the number of connected components of the system \cite{romero2003asymptotic}. 

The cycles of the components of an FG represent the asymptotic behaviour of the system being modelled and they are often called \emph{attractors} for this reason. They can be particularly important when the dynamics modelled by the graph arise from biological applications. As an example, attractors of Boolean networks have been linked to biological phenotypes, making them a crucial factor in the analysis of these models~\cite{SCHWAB2020571}. The idea of trying to understand whether a deterministic discrete phenomenon is decomposable and/or factorisable into smaller dynamics is a problem also studied in the world of DDSs \cite{kadelka2022decomposition, naquin2022factorisation}. A newly developed research direction tries to split an FG (\ie, DDS) into simpler ones by adopting a suitable algebraic approach. Indeed, it has been proven that DDSs equipped with sum (\ie, disjoint union of components) and product operations (\ie, direct product of graphs) form a commutative semiring~\cite{dorigatti2018, weichsel1962kronecker}. 
Thanks to this algebraic setting, one can write polynomial equations in which coefficients and unknowns are FGs. Then, to simplify a certain FG (or to model a hypothesis about the decomposability of a FG), one needs to solve a polynomial equation having the FG as its constant right-hand side. If a solution exists, it can be combined with the coefficients of the equation to obtain the original dynamics through a composition of smaller FGs.

According to the state of the art, to solve these equations, one needs to be able to solve a finite number of \emph{basic} linear equations of the form $A \times X=B$, where $A$ and $B$ are known FGs and $X$ is unknown~\cite{dennunzio2019solving}. To the best of our knowledge, a technique has been introduced to solve these equations, and more general ones, only over DDSs without transient states~\cite{dennunzio2019solving, formenti2021mdds}. 
For this reason, this work aims to study basic linear equations over arbitrary connected FGs (namely $A$, $X$, and $B$) including transient states. 
The idea is to exploit the connection with the classical \emph{cancellation problem} on graphs, which aims at establishing under which conditions the isomorphism of $A\diamond X$ and $A\diamond Y$ implies that $X$ and $Y$ are isomorphic (for a graph product $\diamond$)~\cite{hammack2011handbook}. The cancellation problem has been investigated for the Cartesian, the strong, and the direct products over different classes of graphs (oriented and not oriented, and admitting self-loops or not)~\cite{lovasz1971cancellation, hammack2009direct, hammack2011handbook}. From the literature, we know that cancellation holds, under certain conditions, for the Cartesian and the strong products (over graphs and digraphs), but it can fail for the direct product. Moreover, it is known that this problem is more challenging over digraphs than over graphs. Some questions are still open concerning the cancellation problem for the direct product over digraphs. Some work has been done to characterise \emph{zero divisors} (digraphs for which the cancellation fails), or to find all digraphs $Y$ such that $A \times X \iso A \times Y$, with $A$ and $X$ fixed and $\iso$ is equality up to isomorphisms~\cite{hammack2014zero}. However, the results obtained are concerned with generic digraphs, while this paper focuses on connected digraphs with outgoing degree one to have further interesting results. The interest of this research direction is also confirmed by the fact that, in parallel to the present work, Naquin and Gadouleau~\cite{naquin2022factorisation} also studied the factorisation of FGs and their cancellation properties. In particular, they investigated the cancellation over connected FGs, the uniqueness of the factorisation over a specific class of FGs (namely the connected ones having a fixed point) as well as a way to compute it, and they discuss the uniqueness of $n$-th roots.

In this paper, we show that cancellation holds over structures suitable for modelling FGs, namely  
infinite in-trees (Theorem~\ref{th:isoinfinitetree}), with a novel use of a graph homomorphism counting technique. 
Thanks to this approach, we introduce the first upper bound for the number of solutions of basic linear equations over connected FGs (Theorem~\ref{upperbound}). The existence of such an upper bound is fundamental since it allows us to conceive a polynomial-time algorithm (Sec.~\ref{poly}) which:
\begin{itemize}
    \item either provides precise constraints on potential solutions,
    \item or discards impossible equations when it is not possible to determine such constraints.
\end{itemize}
The ideas and results of the first algorithm are improved and extended to obtain a second one, which is able to provide all solutions (if any) of basic equations (Sec.~\ref{exp}) in exponential time. In order to flatten as much as possible the exponential growth and to reduce the number of considered solutions at each step of the algorithm, we explore some properties of the distance between nodes and their closest cyclic node (also called \emph{depth}), and how this metric behaves with respect to the direct product.
Finally, the performances of the second algorithm are evaluated through experiments. These experiments highlight the exploitability of the algorithm and analyse the dependencies between the performances and different properties of the instances (such as cycle lengths, maximum indegree, and the number of nodes in the graphs).
These algorithms are a fundamental step towards a first approach to solving polynomial equations (with constant right-hand side) over generic FGs.

\section{Functional Digraphs}
\label{sec:FG}

A digraph $\dds = \structure{\setst, E}$ is \emph{functional} (an FG) if and only if each node $\nodedds \in V$ has exactly one outgoing edge. 
In other words, a FG is a graphical representation of a function $f$ from $\setst$ to itself (\ie, an \emph{endofunction}) such that $E=\set{(u,v)\in\setst\times\setst\,\mid\,v=f(u)}$. From now on, when no confusion is possible, the symbol for a generic FG $G$ used interchangeably with the symbol $f$ of function which it represents.
A fundamental property of FGs is that each connected component (or more precisely, \emph{weakly connected component}) has exactly one cycle. For any FG $G$, let $\ncmp$ be the number of its connected components.

A vertex $\nodedds\in \setst$ is a \textbf{cyclic node} of $\dds$ if there exists an integer $p>0$ such that $f^{p}(\nodedds)=\nodedds$. The smallest such $p$ is called the \emph{feedback} (or \emph{period}) of $\nodedds$. 
A \emph{cycle} (of length $p$) of $\dds$ is any set $\cycle=\{\nodedds,f(\nodedds), \ldots, f^{p-1}(\nodedds)\}$ where $\nodedds\in \setst$ is a cyclic node of feedback $p$. 

Given a cycle $\cycle$ and a node $\nodedds\not\in\cycle$, if $\tra$ is the smallest natural number such that $f^\tra(\nodedds) \in \cycle$, we call the set $\{\nodedds,f(\nodedds), \ldots,$ $f^{\tra-1}(\nodedds)\}$ a \textbf{transient of length $\tra$} and its points \textbf{transient nodes}. Then, $\tra$ is called the \emph{depth} of $\nodedds$.

Thus, in FGs,  any outgoing path of any node consists of at most two disjoint parts:  a transient and a cycle.
Clearly, if we denote by  $\pp_\dds$ the set of cyclic vertices and by  $\tp_\dds$ the set of transient ones of $\dds$, we have  $\tp_\dds \cup \pp_\dds=\setst$.

Each component $\dds_j=\structure{\setst_j,E_j}$ with $j\in\set{1,\ldots,\ncmp}$ has one cycle $\cycle_j$ and $\setst_j=\tp_j \cup \cycle_j$. The transient nodes $\tp_j$ of a component $j$ are the nodes $\nodedds \in \tp_\dds$ such that $f^{\tra}(\nodedds)\in\cycle_j$ for some $\tra>0$.

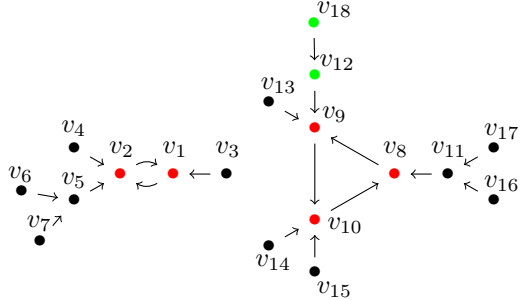
\begin{figure}[ht!]
    \begin{center}
\begin{tikzpicture}
\begin{scope}[scale=0.7]
    \node[draw=none,fill=none, scale=1, anchor=south] (1) at (-1.3,0.15) {$v_1$};
    \node[draw=none,fill=none, scale=1, anchor=south] (2) at (-2.3,0.15) {$v_2$};
    \node[draw=none,fill=none, scale=1, anchor=south] (3) at (-.3,0.15) {$v_3$};
    \node[draw=none,fill=none, scale=1, anchor=south] (4) at (-3.2,0.55) {$v_4$};
    \node[draw=none,fill=none, scale=1, anchor=south] (5) at (-3.2,-0.45) {$v_5$};
    \node[draw=none,fill=none, scale=1, anchor=south] (6) at (-4.2,-0.33) {$v_6$};
    \node[draw=none,fill=none, scale=1, anchor=south] (7) at (-3.85,-1.25) {$v_7$};
    \node[draw=none,fill=none, scale=1, anchor=south] (8) at (2.85,0.15) {$v_8$};
    \node[draw=none,fill=none, scale=1, anchor=south] (9) at (1.7,0.8) {$v_9$};
    \node[draw=none,fill=none, scale=1, anchor=south] (10) at (1.9,-1.3) {$v_{10}$};
    \node[draw=none,fill=none, scale=1, anchor=south] (11) at (3.85,0.15) {$v_{11}$};
    \node[draw=none,fill=none, scale=1, anchor=south] (12) at (1.75,1.8) {$v_{12}$};
    \node[draw=none,fill=none, scale=1, anchor=south] (13) at (0.65,1.35) {$v_{13}$};
    \node[draw=none,fill=none, scale=1, anchor=south] (14) at (0.55,-2) {$v_{14}$};
    \node[draw=none,fill=none, scale=1, anchor=south] (15) at (1.7,-2.5) {$v_{15}$};
    \node[draw=none,fill=none, scale=1, anchor=south] (16) at (4.85,-0.49) {$v_{16}$};
    \node[draw=none,fill=none, scale=1, anchor=south] (17) at (4.85,0.49) {$v_{17}$};
    \node[draw=none,fill=none, scale=1, anchor=south] (18) at (1.7,2.8) {$v_{18}$};
	\begin{oodgraph}
		\addcycle[xshift=-1.83cm, nodes prefix = a,color=red]{2};
		\addbeard[attach node = a->1]{1};
		\addbeard[attach node = a->2]{2};
		\addbeard[attach node = a->2->2]{2};

		\addcycle[xshift=1.83cm, nodes prefix = b,color=red]{3};
		\addbeard[attach node = b->1]{1};
		\addbeard[attach node = b->2,rotation angle=30]{1};
		\addbeard[attach node = b->2,rotation angle=-30,color=green]{1};
		\addbeard[attach node = b->3]{2};
		\addbeard[attach node = b->1->1]{2};
		\addbeard[attach node = b->2->1,color=green,rotation angle=-14]{1};
	\end{oodgraph}
	\end{scope}
\end{tikzpicture}
\end{center}
    \caption{A functional graph $\dds$ with two cycles and some transients nodes connected to both cycles. We have $\dds=(\setst,f)$ with $\setst=\pp_\dds\cup\tp_\dds$, $\pp_\dds=\cycle_1 \cup \cycle_2=\set{v_1,v_2}\cup\set{v_8,v_9,v_{10}}$ and $\tp_\dds=\tp_1 \cup \tp_2=\set{v_3,v_4,v_5,v_6,v_7}\cup\set{v_{11},v_{12},v_{13},v_{14},v_{15},v_{16},v_{17},v_{18}}$. As an example of a transient, we can consider the green nodes. Node $v_{18}$ is contained in transient $\set{v_{18},v_{12}}$ of length $2$.}
    \label{fig:ex_FG}
\end{figure}

\medskip
We consider two algebraic operations over FGs: the \emph{sum} is the disjoint union of the components of two graphs (denoted $+$), and the \emph{product} is the direct product (denoted $\times$) of the graphs.

\begin{definition}[Direct product of digraphs] \label{FGprod}
	Given two FGs, $\dds=(\setst,E)$ and $\dds'=(\setst',E')$, their product $\dds \times \dds'$ is a digraph where the set of nodes is $\setst \times \setst'$ and the set of edges is $\{((\nodedds,\nodedds'),(u,u')) \mid (\nodedds,u) \in E \land (\nodedds',u')$ $\in E'\}$.
\end{definition}

The graph obtained from the direct product corresponds to the representation of the function $(f \times f')$, with $f$ and $f'$ the functions of $\dds$ and $\dds'$, and defined as $(f \times f')(v,u)=(f(v),f'(u))$ for all $(v,u)\in \setst \times \setst'$.

A known property of this operation is that two connected FGs with cycle lengths respectively $|\cycle|$ and $|\cycle'|$ will generate a graph with $\gcd(|\cycle|,|\cycle'|)$ components with cycles of length $\lcm(|\cycle|,|\cycle'|)$~\cite{dennunzio2019solving,hammack2011handbook}. Here, $|S|$ denotes the cardinality of the set $S$. We remark that the result of a direct product of connected FGs is connected if and only if $|\cycle|$ and $|\cycle'|$ are coprime.

\begin{figure}
    \centering
\begin{tikzpicture}[scale=0.7]
    \node[draw=none,fill=none, scale=2] (per) at (3.7,0){\textcolor{gray}{$\times$}};
    \node[draw=none,fill=none, scale=2] (eq) at (11,0){\textcolor{gray}{$=$}};
    
    \node[draw=none,fill=none, scale=0.7, anchor=south] (1) at (-0.5,0.2) {$1$};
    \node[draw=none,fill=none, scale=0.7, anchor=south] (2) at (0.5,0.2) {$2$};
    \node[draw=none,fill=none, scale=0.7, anchor=south] (3) at (-1.35,0.6) {$3$};
    \node[draw=none,fill=none, scale=0.7, anchor=south] (4) at (-1.35,-0.39) {$4$};
    \node[draw=none,fill=none, scale=0.7, anchor=south] (5) at (-2.3,1) {$5$};
    \node[draw=none,fill=none, scale=0.7, anchor=south] (6) at (1.5,0.2) {$6$};
    \node[draw=none,fill=none, scale=0.7, anchor=south] (b2) at (6,0.2) {$b$};
    \node[draw=none,fill=none, scale=0.7, anchor=south] (d2) at (8,0.2) {$d$};
    \node[draw=none,fill=none, scale=0.7, anchor=south] (a2) at (7.35,0.8) {$a$};
    \node[draw=none,fill=none, scale=0.7, anchor=south] (c2) at (7.35,-1.2) {$c$};
    \node[draw=none,fill=none, scale=0.7, anchor=south] (e2) at (7.35,1.8) {$e$};
    \node[draw=none,fill=none, scale=0.7, anchor=south] (f2) at (7.35,2.8) {$f$};
    \node[draw=none,fill=none, scale=0.7, anchor=south] (g2) at (7.35,3.8) {$g$};
    \node[draw=none,fill=none, scale=0.7, anchor=south] (h2) at (9,0.6) {$h$};
    \node[draw=none,fill=none, scale=0.45, anchor=south] (2_a) at (14.55,2.8) {$(2,a)$};
    \node[draw=none,fill=none, scale=0.45, anchor=south] (1_b) at (14,3.4) {$(1,b)$};
    \node[draw=none,fill=none, scale=0.45, anchor=south] (2_c) at (13.45,2.8) {$(2,c)$};
    \node[draw=none,fill=none, scale=0.45, anchor=south] (1_d) at (14,2.2) {$(1,d)$};
    \node[draw=none,fill=none, scale=0.7, anchor=south] (i2) at (9,-1) {$i$};
    \node[draw=none,fill=none, scale=0.45, anchor=south] (1_a) at (14.55,-3.2) {$(1,a)$};
    \node[draw=none,fill=none, scale=0.45, anchor=south] (1_c) at (13.45,-3.2) {$(1,c)$};
    \node[draw=none,fill=none, scale=0.45, anchor=south] (2_b) at (14,-2.6) {$(2,b)$};
    \node[draw=none,fill=none, scale=0.45, anchor=south] (2_d) at (14,-3.8) {$(2,d)$};
    \node[draw=none,fill=none, scale=0.7, anchor=south] (5_f) at (17.7,-3.3) {$(5,f)$};
    \node[draw=none,fill=none, scale=0.7, anchor=south] (5_c) at (16.2,-5.4) {$(5,c)$};
    \node[draw=none,fill=none, scale=0.7, anchor=south] (2_g) at (16.5,-0.4) {$(2,g)$};
	\begin{oodgraph}
		\addcycle[nodes prefix = a]{2};
		\addbeard[attach node = a->1]{1};
		\addbeard[attach node = a->2]{2};
		\addbeard[attach node = a->2->1]{1};
	\end{oodgraph}
	\begin{oodgraph}
		\addcycle[xshift=7cm, nodes prefix = e]{4};
		\addbeard[attach node = e->2]{1};
		\addbeard[attach node = e->1]{2};
		\addbeard[attach node = e->2->1]{1};
		\addbeard[attach node = e->2->1->1]{1};
		\end{oodgraph}
	
	\begin{oodgraph}
		\addcycle[xshift=14cm,yshift=3cm, nodes prefix = b]{4};
		\addbeard[attach node = b->1]{3};
		\addbeard[attach node = b->2]{2};
		\addbeard[attach node = b->3]{1};
		\addbeard[attach node = b->4]{8};
		\addbeard[attach node = b->1->2]{3};
		\addbeard[attach node = b->2->2]{2};
		\addbeard[attach node = b->4->4]{1};
		\addbeard[attach node = b->1->2->3]{1};
		\addbeard[attach node = b->1->2->1]{2};
		\addcycle[xshift=14cm,yshift=-3cm,nodes prefix = c]{4};
		\addbeard[attach node = c->2]{1};
		\addbeard[attach node = c->3]{2};
		\addbeard[attach node = c->4]{5};
		\addbeard[attach node = c->1]{5};
		\addbeard[attach node = c->3->2]{1};
		\addbeard[attach node = c->1->1]{3};
		\addbeard[attach node = c->1->5]{2};
		\addbeard[attach node = c->1->3]{1};
		\addbeard[attach node = c->1->5->2]{3};

	\end{oodgraph}
\end{tikzpicture}
    \caption{An example of the product operation between connected FGs $\dds$ and $\dds'$. For clarity, only the names of some nodes are shown. According to the cycle lengths of the two graphs, the result of the product operation $\dds \times \dds'$ consists of $\gcd(2,4)=2$ components with cycles of length $\lcm(2,4)=4$. Intuitively, we can see the two resulting cycles as the two possible ``parallel" executions of the cyclic dynamics ($\set{(1,a),(2,b),(1,c),(2,d)}$ or $\set{(2,a),(1,b),(2,c),(1,d)}$). The set of vertices of the resulting graph is the Cartesian product $\setst\times\setst'$. A generic node $(v,v')$ is cyclic in the result iff $v$ is a cyclic node of $\dds$ and $v'$ is a cyclic node of $\dds'$, otherwise $(v,v')$ is transient. An alternative way of expressing the transients connected to a cyclic node $(v,v')$ is by retracing the edges of the original graphs backwards (\ie, by considering all possible predecessors). For example, the set of possible predecessors of the node $(1,a)$ is the Cartesian product of predecessors of $1$ and of $a$ (\ie, $\set{2,3,4}$ and $\set{d,e}$, respectively). In fact, we see that $(1,a)$ has $6$ incoming edges. Note that backtracking the edges in the transient nodes will produce a transient of depth equal to the minimum of the original transients depths (see $\set{(5,f),(3,e)}$ for example). However, by backtracking only in cyclic nodes of one graph and only in transient nodes of the other graph, we obtain a copy of the transient (consider $\set{(2,g),(1,f),(2,e)}$ and $\set{(5,c),(3,d)}$ for example).
     }
    \label{fig:prod}
\end{figure}
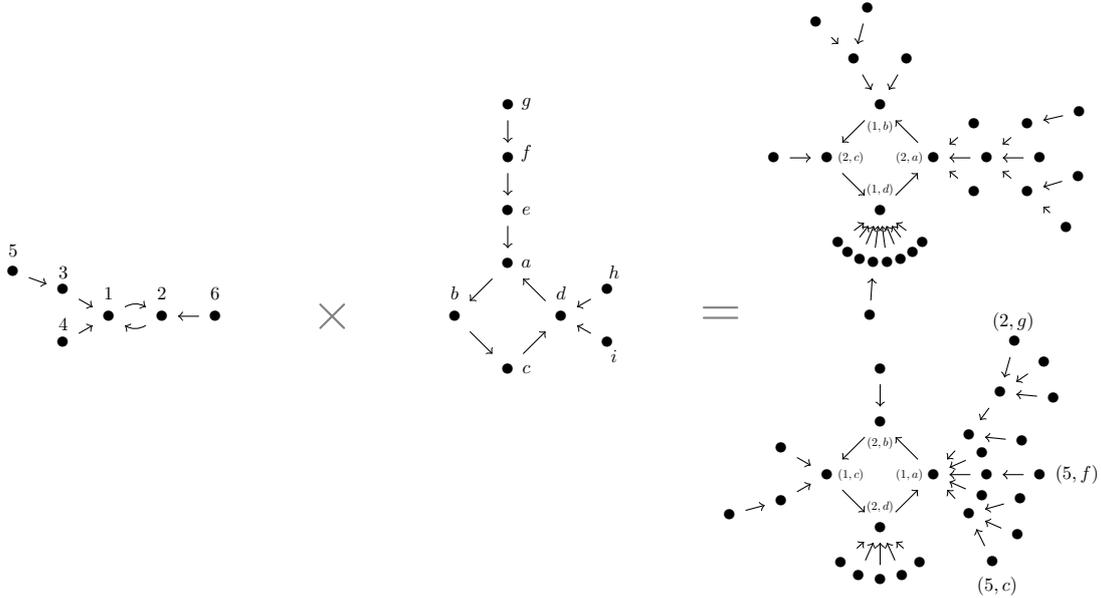

\medskip

It has been proved that the set of FGs (up to isomorphism) equipped with these two operations has the structure of a commutative semiring $\ring$ \cite{dorigatti2018}. This naturally leads to polynomial equations over these graphs which are useful to study decompositions and factorisations of FGs. 

Let us point out that, in the following, we will always consider functional graphs up to isomorphism.

Consider the semiring $\ring[X_1,\ldots,X_\qtvar]$ of polynomials over $\ring$ in the variables $X_1,\ldots, X_\qtvar$, naturally induced by $\ring$.
Unfortunately, the problem of deciding if a solution exists is undecidable for multivariate equations~\cite{dorigatti2018}.
However, when the right-hand side of the equation is constant (\ie, it contains no variables), the problem turns out to be
decidable, but no polynomial-time algorithm able to output the solutions is known.

Considering these computability results and the aim to study decompositions and factorisations, we will be interested in the case of equations with a constant right-hand side and polynomial left-hand side admitting monomials with just one variable. More formally, we will consider equations of the form

\begin{equation} 
	\coef_1\times X_1^{\esp_1} + \cdots + \coef_\mon \times X_\mon^{\esp_\mon}\iso \rterm.
	\label{eq:problemSolvedOriginal}
\end{equation}

In such an equation, for each monomial $z$ (with $1\leq z \leq m$), we denote the coefficient by $\coef_z$ and the variable by $X_z$. Then, we use $\rterm$ to denote the constant term. Let us point out that coefficients, variables, and the constant term are in $\ring$, but the powers $\esp_z \in \N$. Moreover, the $\coef_z$ (and $X_z$) denotes $m$ potentially different FGs and not $m$ components of one graph.
Note that since $\ring$ does not provide a subtraction operation, one
cannot move terms between the two sides of the equation.

In \cite{formenti2021mdds}, a pipeline to enumerate the solutions of Equation \eqref{eq:problemSolvedOriginal} considering just the cycles of coefficients, unknowns and constant right-hand side has been introduced. 
Solutions found are just candidate solutions for the original equation, and on the other hand, solutions
to the original equation must satisfy the restricted equation.

\section{Basic equations over FGs}
\label{sec:transients}

Previous works show that to solve a generic equation over the cyclic part of FGs, it is possible to \emph{enumerate} the solutions of simpler equations
of the form $\coef \times X \iso \rterm$ (with $\coef$, $X$ and $\rterm$ generic FGs)~\cite{dennunzio2019solving}.
Consequently, the aim of this work is to propose an approach to study not only the possible cycles of an FG $X$ but also all its possible transients.

Since the product operation over FGs is distributive over connected components, we restrict our attention to equations with $\coef$ and $X$ made by a single connected component.

Let $B$ also be connected. We will refer to their sets of cyclic nodes as $\cycle_A$, $\cycle_X$ and $\cycle_B$. Note that since $A$, $X$ and $B$ are connected, we will have $\cycle_A=\pp_A$ and similarly for $X$ and $B$. Our goal is then to find all FGs $X$ that, once multiplied by $A$, generate a graph having at least a component isomorphic to $B$.
If $\gcd(|\cycle_A|,|\cycle_X|)=1$, we have $ A \times X \iso B$, and if $\gcd(|\cycle_A|,|\cycle_X|)>1$, we will write $A \times X \supset B$ (to be understood as $B$ is a subgraph of $A \times X$) because $B$ is just one of the connected components of the result of the product operation.
Then, in general, we write
\begin{equation}\label{DDSineq}
    A \times X \supseteq B
\end{equation} 
to consider both scenarios.

\subsection{The t-abstraction}\label{notation}

%
By introducing an abstraction (called the \emph{t-abstraction}) over the FGs $A$, $X$, and $B$, we obtain new equations, which allow us to compute some information about the transient behaviour of the possible solutions of $X$.
More intuitively, these solutions give us strong constraints over the actual solutions of Equation~\eqref{DDSineq}, which drastically reduce the number of possible graphs.

\medskip 

Let us now consider a generic FG $G$ with $\ncmp$ components. First of all, let us fix an arbitrary indexing function for cycles.
Given a cycle $\cycle_j$ (for $j\in\set{1,\ldots ,\ncmp}$), fix a
node $v\in\cycle_j$ and define the \emph{index function} $\ind:\set{0,\dots,|\cycle_j|-1}\longrightarrow\cycle_j$ such that $\ind(\indexp)=f^{\indexp}(\nodedds)$.
We now denote $\tpmatrix{\dds_j}{\indexp,\tra}$ the set of transient nodes of
$\setst_j$ for which there exists a path of length $\tra$ ending in $\ind(\indexp)$ and made only of transients nodes except for the last one. From now on, we will refer to this set as \emph{the nodes at the layer} (or level) $\tra$ for the cyclic vertex $\ind(\indexp)$. We remark that $j$, $\indexp$ and $\tra$ respect specific ranges, namely $j\in\set{1,\ldots,\ncmp}$,
$\indexp\in\set{0,\ldots,|\cycle_j|-1}$, and $\tra\in\set{1,\ldots,\tra^{\max}_j}$, where $\tra^{\max}_j$ is the maximum length of a transient in the $j$-th component.
These sets can be conveniently represented by a matrix 
$\tpmatrix{\dds_j}{}$ of sets in which an element at position $(r,h)$ is
given by
\[
\tpmatrix{\dds_j}{\indexp,\tra}=\begin{cases}f^{-1}(\ind(\indexp))\setminus\cycle_j & \text{if $h=1$}\\
\bigcup_{\nodedds\in\tpmatrix{\dds_j}{\indexp,\tra-1}} f^{-1}(\nodedds) & \text{if $h>1$.}\end{cases}
\]
We remark that $\tpmatrix{\dds_j}{}$ is a matrix containing the transient nodes of a component $\dds_j$ arranged by their height and first cyclic node in the orbit. This matrix is our starting point to introduce t-abstractions.

\begin{figure}
\centering
\begin{tikzpicture}
\begin{scope}[xshift=-6.5cm, yshift=0.5cm]
    \node (0) at (-1,0) {$\matrixt{\dds}{}=$
    \resizebox{0.4\textwidth}{!}{
$\begin{pNiceMatrix}[first-row,first-col,nullify-dots]

      & 1    & 2 &  3      & 4          \\
0    & \multiset{2} & \multiset{1} & \multiset{0} & \emptyset   \\
1 &\multiset{3} & \multiset{2,2} & \multiset{1,1,1,0} & \multiset{0,0,0} \\
2      & \multiset{3} & \multiset{1,1} & \multiset{3,1} & \multiset{0,0,0,0}     \\
3    & \multiset{5} & \multiset{0,0,1,1} & \multiset{0,0} & \emptyset  
\end{pNiceMatrix}$}};
\node[below=.4 of 0] {$\tpmatrix{\dds}{}=$
\resizebox{0.6\textwidth}{!}{
$\begin{pNiceMatrix}[first-row,first-col,nullify-dots]

      & 1    & 2 &  3           \\
0    & \set{v_{5}} & \set{v_{14}} & \emptyset  \\
1 &\set{v_{6},v_{7}} & \set{v_{15},v_{16},v_{17},v_{18}} & \set{v_{23},v_{24},v_{25}}  \\
2      & \set{v_{8},v_{9}} & \set{v_{19},v_{20}} & \set{v_{26},v_{27},v_{28},v_{29}}     \\
3    & \set{v_{10},v_{11},v_{12},v_{13}} & \set{v_{21},v_{22}} & \emptyset 
\end{pNiceMatrix}$}};

\end{scope}
\begin{scope}[scale=0.8]
    \node[draw=none,fill=none, scale=1, anchor=south] (1) at (0.3,0.85) {$v_1$};
    \node[draw=none,fill=none, scale=1, anchor=south] (5) at (0.3,1.85) {$v_5$};
    \node[draw=none,fill=none, scale=1, anchor=south] (14) at (0.35,2.85) {$v_{14}$};
    \node[draw=none,fill=none, scale=1, anchor=south] (2) at (-1,0.2) {$v_2$};
    \node[draw=none,fill=none, scale=1, anchor=south] (6) at (-1.8,0.6) {$v_6$};
    \node[draw=none,fill=none, scale=1, anchor=south] (15) at (-2.3,1.2) {$v_{15}$};
    \node[draw=none,fill=none, scale=1, anchor=south] (23) at (-3.3,1.7) {$v_{23}$};
    \node[draw=none,fill=none, scale=1, anchor=south] (16) at (-2.8,0.3) {$v_{16}$};
    \node[draw=none,fill=none, scale=1, anchor=south] (24) at (-3.8,0.4) {$v_{24}$};
    \node[draw=none,fill=none, scale=1, anchor=south] (7) at (-1.8,-1.1) {$v_7$};
    \node[draw=none,fill=none, scale=1, anchor=south] (17) at (-2.8,-0.8) {$v_{17}$};
    \node[draw=none,fill=none, scale=1, anchor=south] (25) at (-3.8,-0.9) {$v_{25}$};
    \node[draw=none,fill=none, scale=1, anchor=south] (18) at (-2.3,-1.8) {$v_{18}$};
    \node[draw=none,fill=none, scale=1, anchor=south] (3) at (-0.4,-1.3) {$v_3$};
    \node[draw=none,fill=none, scale=1, anchor=south] (8) at (-0.8,-2) {$v_8$};
    \node[draw=none,fill=none, scale=1, anchor=south] (19) at (-1.2,-3) {$v_{19}$};
    \node[draw=none,fill=none, scale=1, anchor=south] (26) at (-1.9,-3.9) {$v_{26}$};
    \node[draw=none,fill=none, scale=1, anchor=south] (27) at (-1,-4.4) {$v_{27}$};
    \node[draw=none,fill=none, scale=1, anchor=south] (28) at (-0.1,-4.3) {$v_{28}$};
    \node[draw=none,fill=none, scale=1, anchor=south] (20) at (1.2,-3) {$v_{20}$};
    \node[draw=none,fill=none, scale=1, anchor=south] (29) at (1.5,-4) {$v_{29}$};
    \node[draw=none,fill=none, scale=1, anchor=south] (9) at (0.8,-2) {$v_9$};
    \node[draw=none,fill=none, scale=1, anchor=south] (4) at (1,-0.6) {$v_4$};
    \node[draw=none,fill=none, scale=1, anchor=south] (10) at (2.1,-1.2) {$v_{10}$};
    \node[draw=none,fill=none, scale=1, anchor=south] (11) at (2.4,-0.7) {$v_{11}$};
    \node[draw=none,fill=none, scale=1, anchor=south] (12) at (2.35,-0.25) {$v_{12}$};
    \node[draw=none,fill=none, scale=1, anchor=south] (21) at (3.4,0) {$v_{21}$};
    \node[draw=none,fill=none, scale=1, anchor=south] (13) at (2.1,0.3) {$v_{13}$};
    \node[draw=none,fill=none, scale=1, anchor=south] (22) at (3.1,0.7) {$v_{22}$};
	\begin{oodgraph}
		\addcycle[nodes prefix = a,color=red]{4};
		\addbeard[attach node = a->1]{4};
		\addbeard[attach node = a->2]{1};
		\addbeard[attach node = a->3]{2};
		\addbeard[attach node = a->4]{2};
		\addbeard[attach node = a->1->3]{1};
		\addbeard[attach node = a->1->4]{1};
		\addbeard[attach node = a->2->1]{1};
		\addbeard[attach node = a->3->1]{2};
		\addbeard[attach node = a->3->2]{2};
		\addbeard[attach node = a->4->1]{1};
		\addbeard[attach node = a->4->2]{1};
		\addbeard[attach node = a->3->1->1]{1};
		\addbeard[attach node = a->3->1->2]{1};
		\addbeard[attach node = a->3->2->1]{1};
		\addbeard[attach node = a->4->1->1]{3};
		\addbeard[attach node = a->4->2->1]{1};
	\end{oodgraph}
	\end{scope}
\end{tikzpicture}
   \caption{We consider the connected FG $\dds$ with a cycle of length 4 on the right. A possible index function is $\ind(y)=\nodedds_{y+1}$. According to $\dds$ and $\ind(y)$, the corresponding t-abstraction $\matrixt{\dds}{}$ and the matrix $\tpmatrix{\dds}{}$ of $\dds$ are shown on the left. In both cases, $\indexp$ runs through rows and $\tra$ through columns. Note that $|\matrixt{\dds}{\indexp,\tra}|$ is always equal to $|\tpmatrix{\dds}{\indexp,\tra-1}|$ since each element of $\matrixt{\dds}{\indexp,\tra}$ represents the number of predecessors of a node in $\tpmatrix{\dds}{\indexp,\tra-1}$.}
    \label{fig:tabs}
\end{figure}
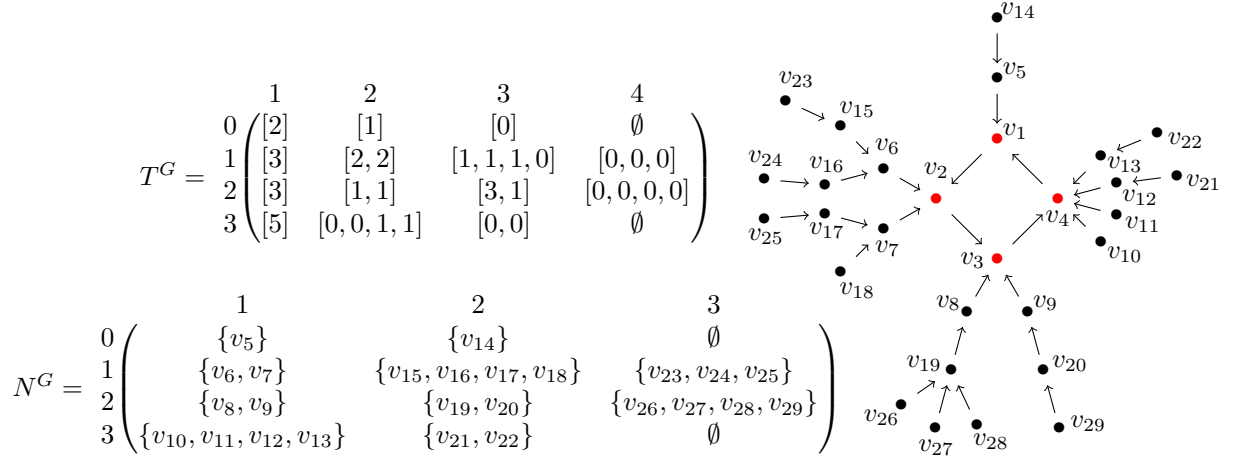

For each component $j$, we now introduce a new matrix $\matrixt{\dds_j}{}$, with $|\cycle_j|$ rows and $\tra^{\max}_j+1$ columns, to model in a new way the transient part of the graph.
Each element of this matrix is a multiset $\matrixt{\dds_j}{\indexp,\tra}$, with $\indexp\in\set{0,\ldots,|\cycle_j|-1}$ and $\tra\in\set{1,\ldots,\tra^{\max}_j+1}$, containing the number of predecessors (\ie, the incoming degree) of each node in $\tpmatrix{\dds_j}{\indexp,\tra-1}$ (\ie, the number of predecessors of each node in the layer $\tra-1$ for the cyclic vertex $\ind(\indexp)$). 
We call \emph{predecessors} of a node $\nodedds$ all those nodes for which there is an outgoing edge towards $\nodedds$.
Then, $\matrixt{\dds_j}{\indexp,\tra}$ is a multiset since there can be two nodes at the same level $\tra-1$ for the same cyclic vertex $\ind(\indexp)$ with an equivalent number of predecessors.
We denote a multiset using square brackets. For example, $\multiset{2,3,3,4}$ denotes the multiset containing symbols $2$ and $4$ with multiplicity $1$, while the symbol $3$ has multiplicity $2$. 

Let us point out that, for $h=1$, we analyse the number of predecessors of the cyclic node $\ind(\indexp)$. Then, we include its predecessor in the cycle $\cycle_j$. 
More formally, for a component $j$, we have
\[
\matrixt{\dds_j}{\indexp,\tra}=\begin{cases}\multiset{|f^{-1}(\ind(\indexp))|}& \text{if $h=1$}\\
\multiset{|f^{-1}(\nodedds)| : \nodedds\in\tpmatrix{\dds_j}{\indexp,\tra-1}} & \text{if $\tra>1.$}\end{cases}
\]


Finally, the \textbf{t-abstraction} of a functional graph $\dds = \dds_1 + \cdots + \dds_{\ncmp}$ (with $\dds_1, \ldots, \dds_{\ncmp}$ connected components) is the multiset $\tabs{\dds}=\multiset{\matrixt{\dds_1}{},\ldots,\matrixt{\dds_\ncmp}{}}$.
In other words, the t-abstraction of a generic FG is the multiset containing the t-abstraction of each of its components. It is imperative to consider multisets, since an FG may contain components that are isomorphic to each other (\ie, with the same t-abstraction). 

Figure \ref{fig:tabs} illustrates an example to clarify all the notations just introduced.

At this point, it must be emphasised that two non-isomorphic connected components may have the same t-abstraction, as illustrated in Figure \ref{fig:2dds_1abs} for example.

 \begin{figure}[ht!]
    \centering
    \begin{tikzpicture}

    \node (0) at (0,0) {
    $\begin{blockarray}{ccccc}
      &            1 &                2 &                3 &            4 \\
      \begin{block}{c(cccc)}
    0 & \multiset{3} &   \multiset{2,1} & \multiset{0,0,1} & \multiset{0} \\
    1 & \multiset{4} & \multiset{0,0,2} &   \multiset{0,0} &    \emptyset \\
    \end{block}
    \end{blockarray}$
    };
    
    \begin{oodgraph}
        \addcycle[xshift=-3.5cm,yshift=2cm,nodes prefix = a]{2};
		\addbeard[attach node = a->1]{2};
		\addbeard[attach node = a->1->1]{1};
		\addbeard[attach node = a->1->2]{2};
		\addbeard[attach node = a->1->2->1]{1};
		\addbeard[attach node = a->2]{3};
		\addbeard[attach node = a->2->2]{2};
		
		\addcycle[xshift=3.5cm,yshift=2cm,nodes prefix = b]{2};
		\addbeard[attach node = b->1]{2};
		\addbeard[attach node = b->1->1]{1};
		\addbeard[attach node = b->1->2]{2};
		\addbeard[attach node = b->1->1->1]{1};
		\addbeard[attach node = b->2]{3};
		\addbeard[attach node = b->2->2]{2};
    \end{oodgraph}
    
    \end{tikzpicture}
    \caption{Two nonisomorphic graphs (top) having the same t-abstractions (bottom).}
    \label{fig:2dds_1abs}
 \end{figure}
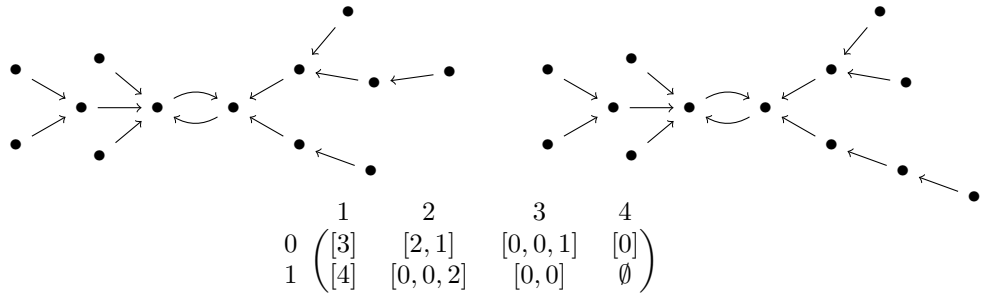




\subsubsection{Products of t-abstractions}

To solve instances of Equation \eqref{DDSineq}, we consider the corresponding equation over t-abstractions $\matrixt{A}{} \times \matrixt{X}{}$ $\supseteq$ $\matrixt{B}{}$, where $\matrixt{A}{}$, $\matrixt{X}{}$ and $\matrixt{B}{}$ are the t-abstractions of respectively $A$, $X$ and $B$ since they are connected.
Here, we study how the product of the t-abstractions of two connected FGs can be computed to obtain the t-abstraction of the product between the original FGs. 
We consider connected FGs because, in the case of a product operation between two generic functional graphs $\tabs{\dds}=\multiset{\matrixt{\dds_1}{},\ldots,\matrixt{\dds_{\ncmp_\dds}}{}}$ and $\tabs{\ddsps}=\multiset{\matrixt{\ddsps_1}{},\ldots,\matrixt{\ddsps_{\ncmp_\ddsps}}{}}$, the result can in fact be computed as $\sum_{z=1}^{\ncmp_\dds} \sum_{z'=1}^{\ncmp_\ddsps} \matrixt{\dds_z}{} \times \matrixt{{\ddsps_{z'}}}{}$.
We denote the number of components of $\dds$ and $\ddsps$ by $\ncmp_\dds$ and $\ncmp_\ddsps$ respectively.

For two multisets of natural numbers $M=\multiset{d_1,\ldots,d_n}$ and $M'=\multiset{d'_1,\ldots,d'_{m}}$, we denote their disjoint union as $M + M'$ and  we define the product $M \otimes M'=\multiset{d \dot d' \mid d\in M, d'\in M'}$. 

\begin{proposition}\label{propprodgeneral}
    Consider two connected FGs $\dds$ and $\ddsps$ with their respective t-abstractions $\tabs{\dds}= \multiset{\matrixt{\dds}{}}$ (with $|\cycle_\dds|$ rows and $\tra^{\max}_\dds$ columns) and $\tabs{\ddsps}= \multiset{\matrixt{\ddsps}{}}$ (with $|\cycle_\ddsps|$ rows and $\tra^{\max}_\ddsps$ columns). Then, the t-abstraction $\tabs{\ddspps}$ of the product $\ddspps = \dds \times \ddsps$ is a multiset of matrices
    $[\matrixt{\ddspps_1}{}, \ldots, \matrixt{\ddspps_{\ncmp_\ddspps}}{}]$ with $\ncmp_\ddspps=\gcd(|\cycle_\dds|,|\cycle_\ddsps|)$.
    Each matrix $\matrixt{\ddspps_i}{}$ (with $i\in\set{1,\ldots,\ncmp_\ddspps}$) has $\lcm(|\cycle_\dds|,|\cycle_\ddsps|)$ rows and $\max\set{\tra^{\max}_\dds,\tra^{\max}_\ddsps}$ columns,
    and each element is computed from $\tabs{\dds}$ and $\tabs{\ddsps}$ as  
	\begin{equation} \label{prodnotconnected}
	    \matrixt{\ddspps_i}{r,h}=\matrixt{\dds}{r,h}\otimes\left(\sum_{j=0}^{h-1} \matrixt{\ddsps}{r-j+(i-1),h-j}\right) + \matrixt{\ddsps}{r+(i-1),h}\otimes\left(\sum_{j=1}^{h-1} \matrixt{\dds}{r-j,h-j}\right) 
	\end{equation}
	where the row indices $r$ and $r-j$ (respectively $r-j+(i-1)$ and $r+(i-1)$) are interpreted modulo $|\cycle_\dds|$ (respectively $|\cycle_\ddsps|$).
\end{proposition}

\begin{proof}
In~\cite{dennunzio2019solving} it has been proven
that the product operation between two connected FGs gives $\gcd(|\cycle_\dds|,|\cycle_\ddsps|)$ components with cycle length $\lcm(|\cycle_\dds|,|\cycle_\ddsps|)$.
At this point, we need to prove that the maximum transient length of the resulting components is $\max\set{\tra^{\max}_\dds,\tra^{\max}_\ddsps}$. 

Let us consider $\dds=\structure{\setst_\dds,E_\dds}$ and $\ddsps=\structure{\setst_\ddsps,E_\ddsps}$ with just one component each and with corresponding functions $f_\dds$ and $f_\ddsps$. According to Section~\ref{sec:FG}, we know that $f_\dds^{\tra^{\max}_\dds}(\nodedds)\in\pp_\dds$, for all $\nodedds\in\setst_\dds$, otherwise, with a smaller number of applications of $f_\dds$, the resulting node can be a transient node. Let us now suppose $\tra^{\max}_\dds\geq \tra^{\max}_\ddsps$. In this scenario, for all $h\geq\tra^{\max}_\ddsps$, we have $f_\ddsps^{h}(u)\in\pp_\ddsps$ for all $u\in\setst_\ddsps$.
Then, a node $(\nodedds_1,\nodedds_2)$ in the result of the product operation (with $\nodedds_1\in\tpmatrix{\dds}{\indexp_1,\tra^{\max}_\dds}$ for some $\indexp_1\in\set{0,\ldots,|\cycle_\dds|-1}$) is a transient node that requires $\tra^{\max}_\dds$ applications of $f_\ddspps = f_\dds \times f_\ddsps$ to reach a cyclic node. Moreover, for all $\nodedds_2\in\setst_\ddsps$, we have $(f_\dds \times f_\ddsps)^{\tra^{\max}_\dds}(\nodedds_1,\nodedds_2)\in\pp_{\ddspps}$ and $(f_\dds \times f_\ddsps)^{h}(\nodedds_1,\nodedds_2)\not\in\pp_{\ddspps}$ with $h<\tra^{\max}_\dds$. For this reason, $\tra^{\max}_{\ddspps}=\tra^{\max}_\dds$. We remark that, since the statement is true for all $\nodedds_2$, all components $\tpmatrix{\ddspps_i}{}$ will have maximal transient length $\tra^{\max}_\dds$.
Symmetrically, if we consider $\tra^{\max}_\dds\leq \tra^{\max}_\ddsps$, one obtains all $\tpmatrix{\ddspps_i}{}$ having $\tra^{\max}_\ddsps$ columns, for all $1\leq i\leq \gcd(|\cycle_\dds|,|\cycle_\ddsps|)$.

Let $\ind_\dds$, $\ind_\ddsps$ and $\ind_\ddspps$ be the index functions of the three graphs. By considering a node $(\nodedds_1,\nodedds_2)\in\setst_\dds \times \setst_\ddsps$, if $(\nodedds_1,\nodedds_2)\in\tpmatrix{\ddspps_i}{\indexp,h}$ and $(f_\dds \times f_\ddsps)^h(\nodedds_1,\nodedds_2)=(p_1,p_2)$, we know that $\ind_\ddspps(\indexp)=(p_1,p_2)$ and at least one of the following conditions holds:
\begin{enumerate}[(i)]
    \item $\nodedds_1\in\tpmatrix{\dds}{\indexp_1,h}$ with $\ind_\dds(\indexp_1)=p_1$;
    \item $\nodedds_2\in\tpmatrix{\ddsps}{\indexp_2,h}$ with $\ind_\ddsps(\indexp_2)=p_2$.
\end{enumerate} 
First, let us consider $\nodedds_1\in\tpmatrix{\dds}{\indexp_1,h}$ with $\ind_\dds(\indexp_1)=p_1$. Then, if we are interested in the elements of $\matrixt{\ddspps_i}{r,h}$ including $\nodedds_1$, we need to consider all feasible values of $\nodedds_2$. In order for $(\nodedds_1,\nodedds_2)$ to belong to $\tpmatrix{\ddspps_i}{r,h}$, node $\nodedds_2$ must satisfy $f_\ddsps^h(\nodedds_2)=p_2$. Therefore, $\nodedds_2$ must necessarily belong to one of $\tpmatrix{\ddsps}{r_2,h}$, $\tpmatrix{\ddsps}{r_2-1,h-1}$, $\tpmatrix{\ddsps}{r_2-2,h-2}$, $\ldots$, $\tpmatrix{\ddsps}{r_2-(h-1),1}$ or be the cyclic vertex $\ind_\ddsps(r_2-h)$.
%
The reasoning is similar for the case when $\nodedds_2\in\tpmatrix{\ddsps}{\indexp_2,h}$. Since either $\nodedds_1\in\tpmatrix{\dds}{\indexp_1,h}$, or $\nodedds_2\in\tpmatrix{\ddsps}{\indexp_2,h}$, or both, the predecessors (in terms of multiplicity) can be computed, for the component $i$ containing $(p_1,p_2)$, by
\[\matrixt{\ddspps_i}{r,h}=\matrixt{\dds}{r,h} \otimes
\left( \sum_{j=1}^{h-1} \matrixt{\ddsps}{r-j,h-j} \right) +
\matrixt{\ddsps}{r,h} \otimes
\left( \sum_{j=1}^{h-1} \matrixt{\dds}{r-j,h-j} \right) +
\matrixt{\dds}{r,h} \otimes \matrixt{\ddsps}{r,h}. \]
Let us suppose, once again, that $\nodedds_1\in\tpmatrix{\dds}{\indexp_1,h}$. In order to generate all components, we must consider the fact that there might exist a node $\nodedds_2'$ such that $(f_\dds \times f_\ddsps)^h(\nodedds_1,\nodedds_2')=(p_1,f_\ddsps(p_2))$ with $(f_\dds \times f_\ddsps)^{h'}(p_1,p_2)\neq(p_1,f_\ddsps(p_2))$ for all $h'\in\N$; this means that $(\nodedds_1,\nodedds_2)$ and $(\nodedds_1,\nodedds_2')$ belong to different components of $\ddspps$. In order for $(\nodedds_1,\nodedds_2')$ to belong to $\tpmatrix{\ddspps_i}{r,h}$, the node $\nodedds_2'$ must satisfy $f_\ddsps^h(\nodedds_2')=f_\ddsps(p_2)$ with $\ind_\ddsps(r_2)=p_2$. Then, $\nodedds_2'$ must necessarily belong to one of $\tpmatrix{\ddsps}{(r_2+1),h}$, $\tpmatrix{\ddsps}{(r_2+1)-1,h-1}$, $\tpmatrix{\ddsps}{(r_2+1)-2,h-2}$, $\ldots$, $\tpmatrix{\ddsps}{(r_2+1)-(h-1),1}$ or be the cyclic node $g_\ddsps((r_2 + 1)-h)$.

In general, we can consider a node $u$ such that $(f_\dds \times f_\ddsps)^h(\nodedds_1,u)=(p_1,f_\ddsps^{i}(p_2))$ with $1\leq i\leq \gcd(|\cycle_\dds|,|\cycle_\ddsps|)$ and $(f_\dds \times f_\ddsps)^{h'}(p_1,p_2)\neq(p_1,f_\ddsps^i(p_2))$ for all $h'\in\N$. If $(\nodedds_1,u)\in\matrixt{\ddspps_i}{r,h}$, $u$ must be a node such that $f_\ddsps^h(u)=f_\ddsps^i(p_2)$ with $\ind_\ddsps(r_2)=p_2$. Then, $u$ must necessarily belong to one of $\tpmatrix{\ddsps}{r_2+(i-1),h}$, $\tpmatrix{\ddsps}{r_2+(i-1)-1,h-1}$, $\tpmatrix{\ddsps}{r_2+(i-1)-2,h-2}$, \ldots, $\tpmatrix{\ddsps}{r_2+(i-1)-(h-1),1}$ or be the cyclic node $g_\ddsps(r_2+(i-1)-h)$. As a consequence, we can compute each element of the t-abstraction of an arbitrary component $i$, with $1\leq i\leq \gcd(|\cycle_\dds|,|\cycle_\ddsps|)$, as
\[\matrixt{\ddspps_i}{r,h}=
\matrixt{\dds}{r,h} \otimes
\left(\sum_{j=1}^{h-1} \matrixt{\ddsps}{r-j+(i-1),h-j}\right) +
\matrixt{\ddsps}{r+(i-1),h} \otimes 
\left(\sum_{j=1}^{h-1} \matrixt{\dds}{r-j,h-j} \right) +
\matrixt{\dds}{r,h} \otimes \matrixt{\ddsps}{r+(i-1),h}.\]
Since $\ind_\ddsps(r-j+(i-1))$ and $\ind_\ddsps(r+(i-1))$ represent cyclic nodes of $\ddsps$, and thus exhibit a cyclic behaviour, they must be interpreted modulo $|\cycle_\ddsps|$, and similarly $\ind_\dds(r)$ and $\ind_\dds(r-j)$ modulo $|\cycle_\dds|$.
\end{proof}

\begin{figure}
\centerline{\includegraphics[width=0.65\textwidth]{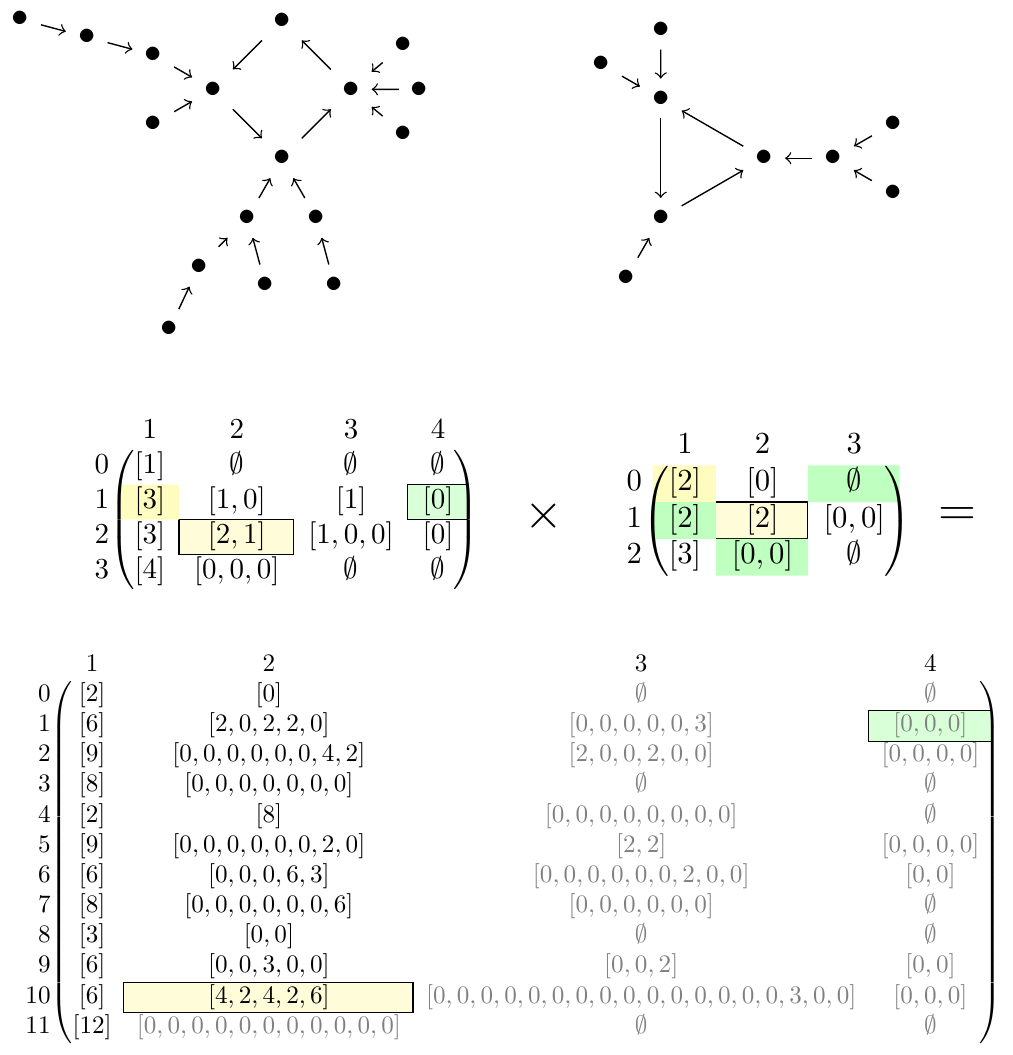}} 
    \caption{An example of product operation between the t-abstractions $\matrixt{\dds}{}$ and $\matrixt{\ddsps}{}$. Above there are the two t-abstractions and the corresponding connected FGs. Below there is the resulting t-abstraction $\matrixt{\ddspps}{}$ computed according to Proposition \ref{propprodgeneral}. Let us suppose that we have already computed all $\matrixt{\ddspps}{\indexp,\tra}$ with $\indexp\in\set{0,\ldots,11}$ and $h=1$, and also all $\matrixt{\ddspps}{\indexp,\tra}$ with $\indexp\in\set{0,\ldots,9}$ and $h=2$. Then, we need to compute $\matrixt{\ddspps}{10,2}$ (in yellow). According to the proposition, $\matrixt{\ddspps}{10,2}=\matrixt{\dds}{2,2}\otimes(\matrixt{\ddsps}{1,2}+\matrixt{\ddsps}{0,1})+\matrixt{\ddsps}{1,2}\otimes\matrixt{\dds}{1,1}=\multiset{2,1}\otimes(\multiset{2,2})+\multiset{2}\otimes\multiset{3}=\multiset{4,2,4,2,6}$. Intuitively, we should follow the diagonal of elements in the two original matrices until we reach the first column (as shown in yellow). Note that if we need to calculate an element $\matrixt{\ddspps}{\indexp,\tra}$ such that one of the two matrices does not have a column $\tra$, we still consider the elements on the diagonal (with $\tra'<\tra$). An example is $\matrixt{\ddspps}{1,4}=\matrixt{\dds}{1,4}\otimes(\matrixt{\ddsps}{0,3}+\matrixt{\ddsps}{2,2}+\matrixt{\ddsps}{1,1})$ (in green in the matrices).
     }
    \label{fig:tprod}
\end{figure}

\subsection{The cancellation problem over transients}\label{cancellation}

In this section, we aim to introduce a first upper bound to the number of solutions of a basic inequality $A \times X \supseteq B$ over connected FGs. We will represent such digraphs with infinite anti-arborescences (or in-trees). An in-tree is a directed rooted tree where the edges are directed towards the root. We call this in-tree the \emph{unroll} of an FG.

\begin{definition}[Unroll of a FG]\label{def:unroll}
Given a connected FG $\dds = \structure{\setst,E}$ representing the function $f$ and a cyclic node $\nodedds$ of $\dds$, the \textbf{unroll} of $\dds$ from $\nodedds$ is an infinite in-tree $\unroll{\dds}{\nodedds}=\structure{\vunroll,\eunroll}$ with vertices $\vunroll=\bigcup_{i\in\N} \vunroll_i$, where $\vunroll_0=\{(\nodedds,0)\}$ and $\vunroll_{i+1}=\{(u,i+1) \mid u \in f^{-1}(v) \text{ for some } (v,i) \in \vunroll_i \}$, and edges between vertices $(u,i+1)$ and $(f(u),i)$ on two consecutive levels $i+1$ and $i$ (with $i\in\N$).
\end{definition}

\begin{figure}
    \begin{center}
\begin{tikzpicture}
\begin{scope}[scale=1]
    \node[draw=none,fill=none, scale=1, anchor=south] (1) at (-1.3,1.15) {$v_1$};
    \node[draw=none,fill=none, scale=1, anchor=south] (2) at (-2.3,1.15) {$v_2$};
    \node[draw=none,fill=none, scale=1, anchor=south] (3) at (-.3,1.15) {$v_3$};
    \node[draw=none,fill=none, scale=1, anchor=south] (4) at (-3.2,1.55) {$v_4$};
    \node[draw=none,fill=none, scale=1, anchor=south] (5) at (-3.2,0.55) {$v_5$};
    \node[draw=none,fill=none, scale=1, anchor=south] (6) at (-4.2,0.77) {$v_6$};
    \node[draw=none,fill=none, scale=1, anchor=south] (7) at (-3.85,-0.75) {$v_7$};
    \begin{oodgraph}
		\addcycle[xshift=-1.8cm, yshift=1cm, nodes prefix = a,color=red]{2};
		\addbeard[attach node = a->1]{1};
		\addbeard[attach node = a->2]{2};
		\addbeard[attach node = a->2->2]{2};
	\end{oodgraph}
	\end{scope}
 \begin{scope}[scale=1.3]
	\node[draw=none,fill=none] (u1) at (5,-1) {$(v_1,0)$};
	\node[draw=none,fill=none] (u2) at (4,-0.2) {$(v_2,1)$};
	\node[draw=none,fill=none] (u2_2) at (6,-0.2) {$(v_3,1)$};
	\node[draw=none,fill=none] (u3) at (3,0.6) {$(v_1,2)$};
	\node[draw=none,fill=none] (u3_2) at (4,0.6) {$(v_4,2)$};
	\node[draw=none,fill=none] (u3_3) at (5,0.6) {$(v_5,2)$};
	\node[draw=none,fill=none] (u4) at (2,1.4) {$(v_2,3)$};
	\node[draw=none,fill=none] (u4_2) at (3,1.4) {$(v_3,3)$};
	\node[draw=none,fill=none] (u4_3) at (4.2,1.4) {$(v_6,3)$};
	\node[draw=none,fill=none] (u4_4) at (5.8,1.4) {$(v_7,3)$};
	\node[draw=none,fill=none, scale=1] (u5) at (1,2.2) {};
	\draw[->,line width=0.3mm] (u2) to (u1);
	\draw[->,line width=0.3mm] (u2_2) to (u1);
	\draw[->,line width=0.3mm] (u3) to (u2);
	\draw[->,line width=0.3mm] (u3_2) to (u2);
	\draw[->,line width=0.3mm] (u3_3) to (u2);
	\draw[->,line width=0.3mm] (u4) to (u3);
	\draw[->,line width=0.3mm] (u4_2) to (u3);
	\draw[->,line width=0.3mm] (u4_3) to (u3_3);
	\draw[->,line width=0.3mm] (u4_4) to (u3_3);
	\draw[->,dotted,line width=0.3mm] (u5) to (u4);
 \end{scope}
\end{tikzpicture}
\end{center}
    \caption{The unroll $\unroll{\dds}{v_1}$ (right) of the connected FG $\dds$ (left). The root $(v_1,0)$ of the infinite in-tree is in the bottom layer of the structure. Intuitively, the unroll shows all the paths which
    eventually reach the root.}
    \label{fig:ex_unroll}
\end{figure}
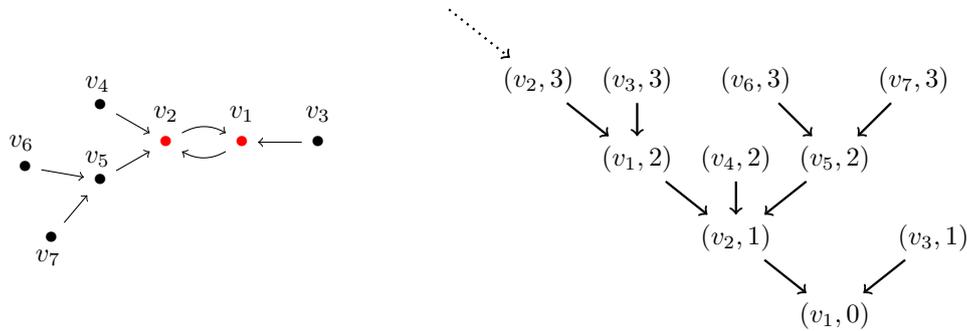

We remark that, starting from a generic connected FG (with just one cycle of length $p$), we can introduce $p$ different unrolls, one for each node $\nodedds\in\pp$.

Let us point out that a similar (but not equivalent) definition of unroll has been independently introduced in \cite{naquin2022factorisation}.

\medskip

Let us introduce a product of (finite or infinite) in-trees to be applied over unrolls to obtain the in-tree modelling the result of a direct product over FGs. Intuitively, this product is the direct product applied layer by layer.

\begin{definition}[Product of in-trees]\label{prodintrees}

Consider two infinite or finite in-trees $\itree_1=\structure{\vunroll_1,\eunroll_1}$ and $\itree_2=\structure{\vunroll_2,\eunroll_2}$ with roots $r_1$ and $r_2$, respectively. The \textbf{product of in-trees} $\itree_1 \unrollp \itree_2$ is the in-tree $\structure{\vunroll,\eunroll}$ such that  $(r_1,r_2)\in\vunroll$ and, for all $(v,u)\in \vunroll$, if there exist $v'\in \vunroll_1$ and $u'\in \vunroll_2$ such that $(v',v)\in\eunroll_1$ and $(u',u)\in\eunroll_2$, then $(v',u')\in\vunroll$. The set of edges is defined as $\eunroll=\set{\structure{\structure{v',u'},\structure{v,u}} \mid \structure{v',v}\in\eunroll_1 \land \structure{u',u}\in\eunroll_2}$.
Notice that $\itree_1 \unrollp \itree_2$ is an infinite in-tree iff $\itree_1$ and $\itree_2$ are infinite in-trees.
\end{definition}

We remark that unrolls of FGs always have one and only one infinite path starting from the root.

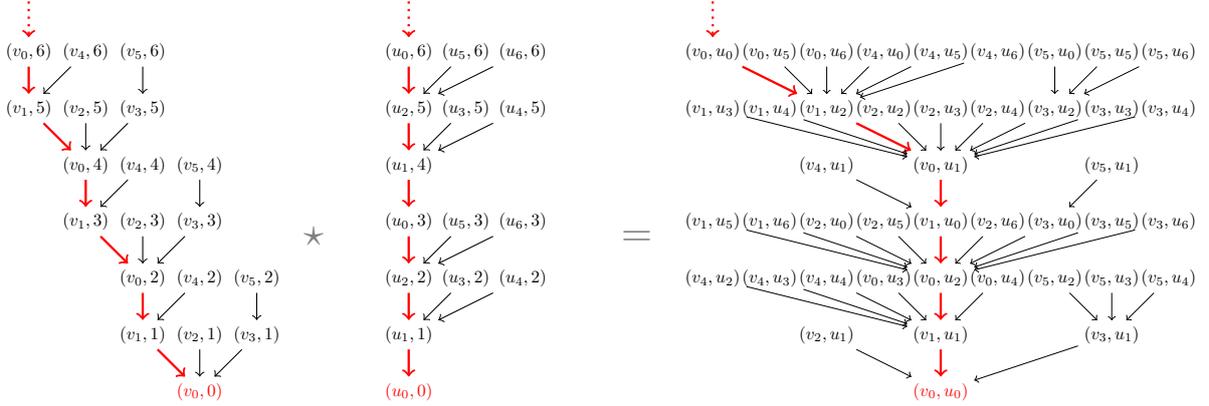
\begin{figure}
    \begin{center}
    
\begin{tikzpicture}
\node[draw=none,fill=none, scale=1.5, anchor=south] (per) at (3,1){\textcolor{gray}{$\unrollp$}};
    \node[draw=none,fill=none, scale=1.5, anchor=south] (eq) at (7.25,1){\textcolor{gray}{$=$}};
\tikzset{
    every node/.style={scale=0.65},
}
\begin{scope}[scale=0.75]
	\node[draw=none,fill=none] (0) at (2,-1) {\textbf{\textcolor{red}{$(v_0,0)$}}};
	\node[draw=none,fill=none] (2) at (2,0) {$(v_2,1)$};
	\node[draw=none,fill=none] (1) at (1,0) {$(v_1,1)$};
	\node[draw=none,fill=none] (3) at (3,0) {$(v_3,1)$};
	\node[draw=none,fill=none] (4) at (1,1) {$(v_0,2)$};
	\node[draw=none,fill=none] (7) at (3,1) {$(v_5,2)$};
	\node[draw=none,fill=none] (5) at (2,1) {$(v_4,2)$};
	\node[draw=none,fill=none] (8) at (0,2) {$(v_1,3)$};
	\node[draw=none,fill=none] (9) at (1,2) {$(v_2,3)$};
	\node[draw=none,fill=none] (10) at (2,2) {$(v_3,3)$};
	\node[draw=none,fill=none] (11) at (0,3) {$(v_0,4)$};
	\node[draw=none,fill=none] (14) at (1,3) {$(v_4,4)$};
	\node[draw=none,fill=none] (16) at (2,3) {$(v_5,4)$};
	\node[draw=none,fill=none] (12) at (-1,4) {$(v_1,5)$};
	\node[draw=none,fill=none] (17) at (0,4) {$(v_2,5)$};
	\node[draw=none,fill=none] (18) at (1,4) {$(v_3,5)$};
	\node[draw=none,fill=none] (13) at (-1,5) {$(v_0,6)$};
	\node[draw=none,fill=none] (19) at (0,5) {$(v_4,6)$};
	\node[draw=none,fill=none] (21) at (1,5) {$(v_5,6)$};
	\node[draw=none,fill=none] (dot) at (-1,6) {};
	\draw[->,color=red,line width=0.3mm] (1) to (0);
	\draw[->] (2) to (0);
	\draw[->] (3) to (0);
	\draw[->,color=red,line width=0.3mm] (4) to (1);
	\draw[->] (5) to (1);
	\draw[->] (7) to (3);
	\draw[->,color=red,line width=0.3mm] (8) to (4);
	\draw[->] (9) to (4);
	\draw[->] (10) to (4);
	\draw[->,color=red,line width=0.3mm] (11) to (8);
	\draw[->] (14) to (8);
	\draw[->,color=red,line width=0.3mm] (12) to (11);
	\draw[->,color=red,line width=0.3mm] (13) to (12);
	\draw[->,dotted,color=red,line width=0.3mm] (dot) to (13);
	\draw[->] (16) to (10);
	\draw[->] (17) to (11);
	\draw[->] (18) to (11);
	\draw[->] (19) to (12);
	\draw[->] (21) to (18);
\end{scope}
\begin{scope}[xshift=5cm,scale=0.75]
	\node[draw=none,fill=none] (0) at (-1,-1) {\textbf{\textcolor{red}{$(u_0,0)$}}};
	\node[draw=none,fill=none] (1) at (-1,0) {$(u_1,1)$};
	\node[draw=none,fill=none] (4) at (-1,1) {$(u_2,2)$};
	\node[draw=none,fill=none] (5) at (0,1) {$(u_3,2)$};
	\node[draw=none,fill=none] (6) at (1,1) {$(u_4,2)$};
	\node[draw=none,fill=none] (8) at (-1,2) {$(u_0,3)$};
	\node[draw=none,fill=none] (9) at (0,2) {$(u_5,3)$};
	\node[draw=none,fill=none] (10) at (1,2) {$(u_6,3)$};
	\node[draw=none,fill=none] (11) at (-1,3) {$(u_1,4)$};
	\node[draw=none,fill=none] (12) at (-1,4) {$(u_2,5)$};
	\node[draw=none,fill=none] (17) at (0,4) {$(u_3,5)$};
	\node[draw=none,fill=none] (18) at (1,4) {$(u_4,5)$};
	\node[draw=none,fill=none] (13) at (-1,5) {$(u_0,6)$};
	\node[draw=none,fill=none] (19) at (0,5) {$(u_5,6)$};
	\node[draw=none,fill=none] (20) at (1,5) {$(u_6,6)$};
	\node[draw=none,fill=none] (dot) at (-1,6) {};
	\draw[->,color=red,line width=0.3mm] (1) to (0);
	\draw[->] (6) to (1);
	\draw[->,color=red,line width=0.3mm] (4) to (1);
	\draw[->] (5) to (1);
	\draw[->,color=red,line width=0.3mm] (8) to (4);
	\draw[->] (9) to (4);
	\draw[->] (10) to (4);
	\draw[->,color=red,line width=0.3mm] (11) to (8);
	\draw[->,color=red,line width=0.3mm] (12) to (11);
	\draw[->,color=red,line width=0.3mm] (13) to (12);
	\draw[->,dotted,color=red,line width=0.3mm] (dot) to (13);
	\draw[->] (17) to (11);
	\draw[->] (18) to (11);
	\draw[->] (19) to (12);
	\draw[->] (20) to (12);
\end{scope}
\begin{scope}[xshift=10.5cm,scale=0.75]
	\node[draw=none,fill=none] (0) at (1,-1) {\textbf{\textcolor{red}{$(v_0,u_0)$}}};
	\node[draw=none,fill=none] (1) at (1,0) {$(v_1,u_1)$};
	\node[draw=none,fill=none] (2) at (-1,0) {$(v_2,u_1)$};
	\node[draw=none,fill=none] (3) at (4,0) {$(v_3,u_1)$};
	\node[draw=none,fill=none] (4) at (1,1) {$(v_0,u_2)$};
	\node[draw=none,fill=none] (a) at (0,1) {$(v_0,u_3)$};
	\node[draw=none,fill=none] (b) at (2,1) {$(v_0,u_4)$};
	\node[draw=none,fill=none] (c) at (-3,1) {$(v_4,u_2)$};
	\node[draw=none,fill=none] (d) at (-2,1) {$(v_4,u_3)$};
	\node[draw=none,fill=none] (e) at (-1,1) {$(v_4,u_4)$};
	\node[draw=none,fill=none] (i) at (3,1) {$(v_5,u_2)$};
	\node[draw=none,fill=none] (l) at (4,1) {$(v_5,u_3)$};
	\node[draw=none,fill=none] (m) at (5,1) {$(v_5,u_4)$};
	\node[draw=none,fill=none] (8) at (1,2) {$(v_1,u_0)$};
	\node[draw=none,fill=none] (a2) at (-3,2) {$(v_1,u_5)$};
	\node[draw=none,fill=none] (b2) at (-2,2) {$(v_1,u_6)$};
	\node[draw=none,fill=none] (c2) at (-1,2) {$(v_2,u_0)$};
	\node[draw=none,fill=none] (d2) at (0,2) {$(v_2,u_5)$};
	\node[draw=none,fill=none] (e2) at (2,2) {$(v_2,u_6)$};
	\node[draw=none,fill=none] (i2) at (3,2) {$(v_3,u_0)$};
	\node[draw=none,fill=none] (l2) at (4,2) {$(v_3,u_5)$};
	\node[draw=none,fill=none] (m2) at (5,2) {$(v_3,u_6)$};

	\node[draw=none,fill=none] (11) at (1,3) {$(v_0,u_1)$};
	\node[draw=none,fill=none] (a3) at (-1,3) {$(v_4,u_1)$};
	\node[draw=none,fill=none] (b3) at (4,3) {$(v_5,u_1)$};
	\node[draw=none,fill=none] (12) at (-1,4) {$(v_1,u_2)$};
	\node[draw=none,fill=none] (a4) at (-3,4) {$(v_1,u_3)$};
	\node[draw=none,fill=none] (b4) at (-2,4) {$(v_1,u_4)$};
	\node[draw=none,fill=none] (c4) at (0,4) {$(v_2,u_2)$};
	\node[draw=none,fill=none] (d4) at (1,4) {$(v_2,u_3)$};
	\node[draw=none,fill=none] (e4) at (2,4) {$(v_2,u_4)$};
	\node[draw=none,fill=none] (f4) at (3,4) {$(v_3,u_2)$};
	\node[draw=none,fill=none] (g4) at (4,4) {$(v_3,u_3)$};
	\node[draw=none,fill=none] (h4) at (5,4) {$(v_3,u_4)$};
	\node[draw=none,fill=none] (13) at (-3,5) {$(v_0,u_0)$};
	\node[draw=none,fill=none] (a5) at (-2,5) {$(v_0,u_5)$};
	\node[draw=none,fill=none] (b5) at (-1,5) {$(v_0,u_6)$};
	\node[draw=none,fill=none] (c5) at (0,5) {$(v_4,u_0)$};
	\node[draw=none,fill=none] (d5) at (1,5) {$(v_4,u_5)$};
	\node[draw=none,fill=none] (e5) at (2,5) {$(v_4,u_6)$};
	\node[draw=none,fill=none] (f5) at (3,5) {$(v_5,u_0)$};
	\node[draw=none,fill=none] (g5) at (4,5) {$(v_5,u_5)$};
	\node[draw=none,fill=none] (h5) at (5,5) {$(v_5,u_6)$};
	\node[draw=none,fill=none] (dot) at (-3,6) {};
	\draw[->,color=red,line width=0.3mm] (1) to (0);
	\draw[->,line width=0.05mm] (2) to (0);
	\draw[->,line width=0.05mm] (3) to (0);
	\draw[->,color=red,line width=0.3mm] (4) to (1);
	\draw[->,line width=0.05mm] (a) to (1);
	\draw[->,line width=0.05mm] (b) to (1);
	\draw[->,line width=0.05mm] (c) to (1);
	\draw[->,line width=0.05mm] (d) to (1);
	\draw[->,line width=0.05mm] (e) to (1);
	\draw[->,line width=0.05mm] (i) to (3);
	\draw[->,line width=0.05mm] (l) to (3);
	\draw[->,line width=0.05mm] (m) to (3);
	\draw[->,line width=0.05mm] (a2) to (4);
	\draw[->,line width=0.05mm] (b2) to (4);
	\draw[->,line width=0.05mm] (c2) to (4);
	\draw[->,line width=0.05mm] (d2) to (4);
	\draw[->,line width=0.05mm] (e2) to (4);
	\draw[->,line width=0.05mm] (i2) to (4);
	\draw[->,line width=0.05mm] (l2) to (4);
	\draw[->,line width=0.05mm] (m2) to (4);
	\draw[->,color=red,line width=0.3mm] (8) to (4);
	\draw[->,color=red,line width=0.3mm] (11) to (8);
	\draw[->,line width=0.05mm] (a3) to (8);
	\draw[->,line width=0.05mm] (b3) to (i2);
	\draw[->,color=red,line width=0.3mm] (12) to (11);
	\draw[->,line width=0.05mm] (a4) to (11);
	\draw[->,line width=0.05mm] (b4) to (11);
	\draw[->,line width=0.05mm] (c4) to (11);
	\draw[->,line width=0.05mm] (d4) to (11);
	\draw[->,line width=0.05mm] (e4) to (11);
	\draw[->,line width=0.05mm] (f4) to (11);
	\draw[->,line width=0.05mm] (g4) to (11);
	\draw[->,line width=0.05mm] (h4) to (11);
	\draw[->,line width=0.05mm] (a5) to (12);
	\draw[->,line width=0.05mm] (b5) to (12);
	\draw[->,line width=0.05mm] (c5) to (12);
	\draw[->,line width=0.05mm] (d5) to (12);
	\draw[->,line width=0.05mm] (e5) to (12);
	\draw[->,line width=0.05mm] (f5) to (f4);
	\draw[->,line width=0.05mm] (g5) to (f4);
	\draw[->,line width=0.05mm] (h5) to (f4);
	\draw[->,color=red,line width=0.3mm] (13) to (12);
	\draw[->,dotted,color=red,line width=0.3mm] (dot) to (13);
\end{scope}
\end{tikzpicture}
\end{center}
    \caption{An example of product $\unrollp$ (see Definition \ref{prodintrees}) over infinite in-trees (in this case two unrolls $\unroll{G_1}{v_0}$ and $\unroll{G_2}{u_0}$). The product is intuitively equivalent to the direct product applied level by level. The roots of the unrolls and their infinite paths are highlighted in red.}
    \label{fig:prod_unroll}
\end{figure}

Considering an instance of Equation~\eqref{DDSineq}, we can now introduce a corresponding equation over unrolls, and study whether there can be more than one solution. 
For the moment, let us consider generic cyclic nodes of $A$, $X$, and $B$ (let $a$, $x$, and $b$ be respectively these three nodes). Suppose that there exists $X$ such that $\unroll{A}{a} \unrollp \unroll{X}{x}\iso\unroll{B}{b}$, and  $Y$ (not isomorphic to $X$) such that $\unroll{A}{a} \unrollp \unroll{Y}{y}\iso\unroll{B}{b}$ (with $y$ also a generic cyclic node of $Y$). If $\unroll{X}{x} \notiso \unroll{Y}{y}$, then there is a difference at some minimal level between the two infinite in-trees. Given a generic infinite in-tree $\itree$ (such as an unroll), we define the function \textbf{cut} which gives the finite subtree up to the layer $t$ (with $t\in\N$). Then, the result of $\mathrm{cut}_t(\itree)$ is a finite sub-in-tree induced by the set of vertices having at most distance $t$ from the root. 
 We will denote the cut of an unroll $\unroll{A}{a}$, at a level $t$, by $\cut{A}{a,t}$.
 Remark that the cut distributes over the $\unrollp$ product. In our problem, if $\unroll{X}{x} \notiso \unroll{Y}{y}$, then there exists a minimum $t\in\N \backslash \{0\}$ such that $\cut{X}{x,t} \notiso \cut{Y}{y,t}$ and $\cut{X}{x,t-1} \iso \cut{Y}{y,t-1}$.

Given two graphs $G_1$ and $G_2$, denote by $\homset{G_1,G_2}$ the set of homomorphisms between $G_1$ and $G_2$ and let $\homo{G_1,G_2}$
be the cardinality of $\homset{G_1,G_2}$. An important result from graph theory relates the isomorphism of two graphs to their number of incoming homomorphisms.

\begin{theorem}[Lóvasz~\cite{lovasz1971cancellation}]
Two graphs $G_1$ and $G_2$ are isomorphic iff, for all graphs $G$, we have $\homo{G,G_1}=\homo{G,G_2}$.
\end{theorem}

Then, in our case, if $\cut{A}{a,t} \unrollp \cut{X}{x,t}\iso \cut{B}{b,t}$ and $\cut{A}{a,t} \unrollp \cut{Y}{y,t} \iso \cut{B}{b,t}$, we must have
\[
\homo{G,\cut{A}{a,t} \unrollp \cut{X}{x,t}}=\homo{G,\cut{A}{a,t} \unrollp \cut{Y}{y,t}}=\homo{G,\cut{B}{b,t}}
\]
for all $G$. Then, it is necessary to understand how $\homo{G,\cut{A}{a,t} \unrollp \cut{X}{x,t}}$ and $\homo{G,\cut{A}{a,t} \unrollp \cut{Y}{y,t}}$ can be computed in terms of $\homo{G,\cut{A}{a,t}}$, $\homo{G,\cut{X}{x,t}}$, and $\homo{G,\cut{Y}{y,t}}$.

As a convention, we will use $\ftree$ to denote finite in-trees (such as infinite in-trees after a cut operation).
 

\begin{theorem}\label{th:homproduct}
For any graph $G$ and for any pair of finite in-trees $F_1, F_2$ we have
\[
\homo{G,\ftree_1 \unrollp \ftree_2}=\homo{G,\ftree_1} \cdot \homo{G,\ftree_2}.
\]
\end{theorem}

To prove this theorem, we need the following lemma.
 
\begin{lemma}\label{lemma:proj}

For any graph $G$ and for any pair of finite in-trees $F_1, F_2$, if $\onehom \in\homset{G,\ftree_1 \unrollp \ftree_2}$ with $\onehom(u)=(u_1,u_2)$ then $\pi_1 \circ \onehom\in\homset{G,\ftree_1}$ and $\pi_2 \circ \onehom\in\homset{G,\ftree_2}$, where $\pi_1$ and $\pi_2$ are the left and right projections, respectively.
\end{lemma}

\begin{proof}
  Take $\onehom\in\homset{G,\ftree_1 \unrollp \ftree_2}$.
  If $(u,v)$ is an edge of $G$, then $(\onehom(u),\onehom(v))$ must be an edge of $\ftree_1 \unrollp \ftree_2$. Let us suppose that $\onehom(u)=(u_1,u_2)$ and $\onehom(v)=(v_1,v_2)$. From the product definition, the fact that $((u_1,u_2),(v_1,v_2))$ is an edge of $\ftree_1 \unrollp \ftree_2$ implies that $(u_1,v_1)$ is an edge of $\ftree_1$ and $(u_2,v_2)$ is an edge of $\ftree_2$. By applying the projection $\pi_1$, we obtain $(\pi_1 \circ \onehom)(u)=u_1$ and $(\pi_1 \circ \onehom)(v)=v_1$. Then, $\pi_1 \circ \onehom \in\homset{G,\ftree_1}$, and by same reasoning, we obtain $\pi_2 \circ \onehom\in\homset{G,\ftree_2}$.
  \end{proof}
 

  \begin{proof}[Proof of Theorem \ref{th:homproduct}]
  Let $\bijHom: \homset{G,\ftree_1 \unrollp \ftree_2} \to \homset{G,\ftree_1} \times \homset{G,\ftree_2}$ be the function defined by $\bijHom(\onehom) = (\pi_1 \circ \onehom, \pi_2 \circ \onehom)$; by Lemma~\ref{lemma:proj} this is indeed a well-defined function. The goal is to prove that $\bijHom$ is a bijection. First, let us prove the surjectivity of $\bijHom$. Let $\structure{\tau_1,\tau_2}\in\homset{G,\ftree_1} \times \homset{G,\ftree_2}$; we need to find $\onehom\in\homset{G,\ftree_1 \unrollp \ftree_2}$ such that $\bijHom(\onehom)=\structure{\onehom_1,\onehom_2}$. For $\onehom(v)=(\onehom_1(v),\onehom_2(v))$, 
  \[\bijHom(\onehom)(v)=(\pi_1(\onehom(v)),\pi_2(\onehom(v)))=(\pi_1(\onehom_1(v),\onehom_2(v)),\pi_2(\onehom_1(v),\onehom_2(v))) = (\onehom_1(v),\onehom_2(v))\] 
  implies $\bijHom(\onehom)=(\onehom_1,\onehom_2)$. Let us now prove that $\bijHom$ is also injective, \ie, $\bijHom(\onehom)=\bijHom(\onehom')$ implies $\onehom=\onehom'$. We have $\bijHom(\onehom)(v)=(\pi_1(\onehom(v)),\pi_2(\onehom(v)))$ and $\bijHom(\onehom')(v)=(\pi_1(\onehom'(v)),\pi_2(\onehom'(v)))$. If $\bijHom(\onehom)=\bijHom(\onehom')$ then $\pi_1(\onehom(v)) = \pi_1(\onehom'(v))$ and $\pi_2(\onehom(v)) = \pi_2(\onehom'(v))$, which implies $\onehom=\onehom'$.
  \end{proof}


As an immediate consequence of Theorem~\ref{th:homproduct} we have
\begin{align*}
\homo{G,\cut{A}{a,t} \unrollp \cut{X}{x,t}} &= \homo{G,\cut{A}{a,t}} \cdot \homo{G,\cut{X}{x,t}}, \\
\homo{G,\cut{A}{a,t} \unrollp \cut{Y}{y,t}} &= \homo{G,\cut{A}{a,t}} \cdot \homo{G,\cut{Y}{y,t}}.
\end{align*}


The following lemma relates the existence of homomorphisms towards trees to the existence of homomorphisms towards paths.
 
  \begin{lemma}\label{lemma:treepath}
  If $G$ is a graph and $F$ is a finite in-tree, then $\homo{G,F}\neq 0$ iff $\homo{G,P_s}\neq 0$, where $P_s$ is any directed path with $s$ edges and $s$ is the height of $F$.
  \end{lemma}
 
  \begin{proof}
  There always exists a homomorphism from an in-tree $\ftree$ to $P_s$ with $s$ being the height of $\ftree$: one can map all the nodes of the $i$-th level of the tree to the $i$-th node of the path. Conversely, a homomorphism from $P_s$ to $\ftree$ exists by choosing any path of length $s$ in $\ftree$. By composition of homomorphisms, $\homo{G,F}\neq 0$ iff $\homo{G,P_s}\neq 0$.
  \end{proof}
 
 We are studying when $\homo{G,\cut{A}{a,t}}\cdot\homo{G,\cut{X}{x,t}} =\homo{G,\cut{A}{a,t}} \cdot\homo{G,\linebreak\cut{Y}{y,t}}=\homo{G,\cut{B}{b,t}}$ implies $\cut{X}{x,t} \iso \cut{Y}{y,t}$. This is a special case of the well-known \emph{cancellation problem} \cite{hammack2011handbook}.
 In this problem, one needs to decide if $A \diamond X \iso A \diamond Y \iso B$ implies $X \iso Y$, according to a specific definition of the product operation $\diamond$. 
 In our case, we aim at studying cancellation over FGs through the $\unrollp$ product over unrolls. At this point of the reasoning, we can reduce our problem to a corresponding one over finite in-trees.
 
 
 \begin{theorem}\label{th:isofinitetree}
 For any finite in-trees $\ftree$, $X$, and $Y$ of the same height,  $\ftree \unrollp X \iso \ftree \unrollp Y$ implies $X \iso Y$.
 \end{theorem}
 
  \begin{proof}
  According to Theorem \ref{th:homproduct}, since $\ftree \unrollp X \iso \ftree \unrollp Y$, we have $\homo{G,\ftree} \cdot \homo{G,X}=\homo{G,\ftree} \cdot \homo{G,Y}$  for all graphs $G$. If $\homo{G,\ftree}\neq0$, we can divide by it to obtain $\homo{G,X}=\homo{G,Y}$. If $\homo{G,\ftree}=0$, according to Lemma \ref{lemma:treepath}, we have $\homo{G,P_{s}}=0$ where $s$ is the height of $\ftree$. Since $\ftree$, $X$, and $Y$ have the same height, also $\homo{G,X}=0$ and $\homo{G,Y}=0$.  
  \end{proof}
 
Theorem~\ref{th:isofinitetree} over finite in-trees can now be generalized to infinite in-trees, including unrolls.
\begin{theorem}\label{th:isoinfinitetree}
$\itree \unrollp X \iso \itree \unrollp Y$ implies $X \iso Y$, for all infinite in-trees $X,Y$, and $\itree$.
\end{theorem}
 
  \begin{proof}
  If $X \notiso Y$, then there exists a minimum $t\in\N\setminus\set{0}$ such that $\mathrm{cut}_t(X) \notiso \mathrm{cut}_t(Y)$ and $\mathrm{cut}_{t-1}(X) \iso \mathrm{cut}_{t-1}(Y)$. Now, since $\itree \unrollp X \iso \itree \unrollp Y$, we have $\mathrm{cut}_t(\itree \unrollp X) \iso \mathrm{cut}_t(\itree \unrollp Y)$. By the distributivity of the cut operation over products, we have $\mathrm{cut}_t(\itree) \unrollp \mathrm{cut}_t(X)  \iso \mathrm{cut}_t(\itree) \unrollp \mathrm{cut}_t(Y)$. According to Theorem~\ref{th:isofinitetree}, this means that $\mathrm{cut}_t(X) \iso \mathrm{cut}_t(Y)$, which is a contradiction.
  \end{proof}
 
 Up to now, we have proved that, considering two equations over unrolls, $\unroll{A}{a} \unrollp \unroll{X}{x} \iso \unroll{A}{a} \unrollp \unroll{Y}{y} \iso \unroll{B}{b}$ implies $\unroll{X}{x} \iso \unroll{Y}{y}$. This means that given a basic equation of unrolls and given $a$ and $b$ nodes, the solution, if any, is unique. Now, we want to show how many equations over unrolls we need to study to enumerate the solutions of $A \times X \supseteq B$. Note that, in the case of $A \times X$ with $\gcd(|\pp_A|,|\pp_X |)\neq 1$, the result of the product of the unrolls of $A$ and $X$ will be the unroll of only one component of $A \times X$, which is in accordance with Equation \eqref{DDSineq}. 

Let us fix $b$ and study, for every $a\in \pp_A$, the corresponding equation over unrolls. This gives us $|\pp_A|$ equations to study. If we fix $a$, we have $|\pp_B|$ equations to consider. Hence, fixing $b$ is more efficient, since $|\pp_A|<|\pp_B|$. The question is now to determine if it is necessary to try another $b'$. To answer this, we introduce the notion of roll, which is intuitively the opposite of an unroll.

\begin{definition}[Roll of an infinite in-tree]\label{def:roll}
Let $\itree=\structure{\vunroll,\eunroll}$ be an infinite in-tree with root $r$ and only one infinite path $(\ldots, v_n, v_{n-1}, \ldots, v_2, v_1 = r)$, and let $\rho \ge 1$ be an integer. Let $J=(\vunroll,\eunroll')$ with $\eunroll'= (\eunroll \setminus \set{(v_{\rho+1},v_\rho)}) \cup \set{(r,v_\rho)}$. Then $J = \itree' + J'$ where $I'$ is an infinite in-tree and $J'$ is a finite connected FG with a cycle of length $\rho$. We call $J'$ the \emph{roll} (of length $\rho$) of $\itree$, denoted $\roll{\itree}{\rho}$.
\end{definition}

Notice that, for all $a \in \pp_A$, we have $\roll{\unroll{A}{a}}{|\pp_A|} \iso A$. 
That is, rolling up the unroll of $A$ from any of its periodic nodes always yields the initial FG $A$ if the length $|\pp_A|$ is used.   

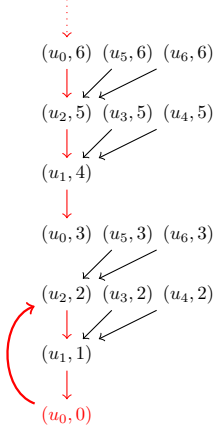
\begin{figure}
    \begin{center}
    
\begin{tikzpicture}
\begin{scope}[xshift=3.5cm,scale=0.8]
	\node[draw=none,fill=none, scale=0.7] (0) at (-1,-1) {\textbf{\textcolor{red}{$(u_0,0)$}}};
	\node[draw=none,fill=none, scale=0.7] (1) at (-1,0) {$(u_1,1)$};
	\node[draw=none,fill=none, scale=0.7] (4) at (-1,1) {$(u_2,2)$};
	\node[draw=none,fill=none, scale=0.7] (5) at (0,1) {$(u_3,2)$};
	\node[draw=none,fill=none, scale=0.7] (6) at (1,1) {$(u_4,2)$};
	\node[draw=none,fill=none, scale=0.7] (8) at (-1,2) {$(u_0,3)$};
	\node[draw=none,fill=none, scale=0.7] (9) at (0,2) {$(u_5,3)$};
	\node[draw=none,fill=none, scale=0.7] (10) at (1,2) {$(u_6,3)$};
	\node[draw=none,fill=none, scale=0.7] (11) at (-1,3) {$(u_1,4)$};
	\node[draw=none,fill=none, scale=0.7] (12) at (-1,4) {$(u_2,5)$};
	\node[draw=none,fill=none, scale=0.7] (17) at (0,4) {$(u_3,5)$};
	\node[draw=none,fill=none, scale=0.7] (18) at (1,4) {$(u_4,5)$};
	\node[draw=none,fill=none, scale=0.7] (13) at (-1,5) {$(u_0,6)$};
	\node[draw=none,fill=none, scale=0.7] (19) at (0,5) {$(u_5,6)$};
	\node[draw=none,fill=none, scale=0.7] (20) at (1,5) {$(u_6,6)$};
	\node[draw=none,fill=none, scale=0.7] (dot) at (-1,6) {};
	\draw[->,color=red] (1) to (0);
	\draw[->] (6) to (1);
	\draw[->,color=red] (4) to (1);
	\draw[->] (5) to (1);
	\draw[->] (9) to (4);
	\draw[->] (10) to (4);
	\draw[->,color=red] (11) to (8);
	\draw[->,color=red] (12) to (11);
	\draw[->,color=red] (13) to (12);
	\draw[->,dotted,color=red] (dot) to (13);
	\draw[->,color=red,line width=0.3mm, bend left=70] (0) to (4);
	\draw[->] (17) to (11);
	\draw[->] (18) to (11);
	\draw[->] (19) to (12);
	\draw[->] (20) to (12);
\end{scope}
\end{tikzpicture}
\end{center}
    \caption{According to Definition \ref{def:roll}, given an infinite in-tree (or an unroll as in this case) and an integer $\rho$ (here equal to $2$), we can define the roll as the connected component with a cycle of length $\rho+1$ obtained by deleting the edge from $(u_0,3)$ to $(u_2,2)$ and adding another one from $(u_0,0)$ to $(u_2,2)$, in this case. }
    \label{fig:roll}
\end{figure}




\begin{lemma}\label{le:isoroll}
Let $A \times X \supseteq B$ with $\lcm(|\pp_A|,|\pp_X|)=|\pp_B|$. Suppose that:
\begin{itemize}
    \item $\unroll{A}{a} \unrollp \unroll{X}{x} \iso \unroll{B}{b}$ for some $a \in \pp_A$, $x \in \pp_X$, and $b \in \pp_B$;
    \item $\unroll{A}{a'} \unrollp \unroll{Y}{y} \iso \unroll{B}{b'}$ for some connected FG $Y$ with $y \in \pp_Y$ and $|\pp_Y|=|\pp_X|$;
    \item $a' = f_A^k(a)$ and $b' = f_B^k(b)$ for some $k \in \N$, where~$f_A^k$ and~$f_B^k$ are the functions represented by FGs~$A$ and~$B$.
\end{itemize}
Then, $X \iso Y$.
\end{lemma}

\begin{proof}
 Let $k'$ be an integer such that $f_B^{k'}(b') = f_B^{k'}(f_B^k(b)) = b$. Since $|\pp_B|$ is a multiple of $|\pp_A|$, this also implies $f_A^{k'}(a') = f_A^{k'}(f_A^k(a)) = a$. This means that $\unroll{B}{b}$ is the unroll of the same FG as in $\unroll{B}{b'}$ but taking the $k'$-th successor of $b$ (\ie, $f_B^{k'}(b)$) as the root, and similarly for $\unroll{A}{a}$ and $\unroll{A}{a'}$. According to the same reasoning, $\unroll{Y}{y'}$ with $y' = f_Y^{k'}(y)$ satisfies $\unroll{A}{a} \unrollp \unroll{Y}{y'} \iso \unroll{B}{b}$. By Theorem~\ref{th:isoinfinitetree}, we have $\unroll{Y}{y'} \iso \unroll{X}{x}$, and so $\roll{\unroll{Y}{y'}}{|\pp_Y|} \iso \roll{\unroll{X}{x}}{|\pp_X|}$ implies $Y\iso X$.
 \end{proof}
 
\begin{figure}
    \begin{center}
    
\begin{tikzpicture}
\node[draw=none,fill=none, scale=2] (per) at (3,0){\textcolor{gray}{$\times$}};
    \node[draw=none,fill=none, scale=2] (eq) at (7.7,0){\textcolor{gray}{$=$}};
\begin{scope}[scale=1]
    \node[draw=none,fill=none, scale=1, anchor=south] (1) at (0.5,0.15) {$v_0$};
    \node[draw=none,fill=none, scale=1, anchor=south] (1) at (-0.5,0.15) {$v_1$};
	\begin{oodgraph}
		\addcycle[nodes prefix = a,color=red]{2};
		\addbeard[attach node = a->1]{2};
		\addbeard[attach node = a->2]{1};
		\addbeard[attach node = a->1->1]{1};
	\end{oodgraph}
\end{scope}
\begin{scope}[xshift=5cm,scale=1]
\node[draw=none,fill=none, scale=1, anchor=south] (1) at (-0.8,-1) {$u_0$};
    \node[draw=none,fill=none, scale=1, anchor=south] (1) at (-0.7,0.5) {$u_1$};
    \node[draw=none,fill=none, scale=1, anchor=south] (1) at (1,0.15) {$u_2$};
	\begin{oodgraph}
		\addcycle[nodes prefix = a,color=red]{3};
		\addbeard[attach node = a->1]{2};
		\addbeard[attach node = a->2]{2};
	\end{oodgraph}
\end{scope}
\begin{scope}[xshift=11cm,scale=1.1]
\node[draw=none,fill=none, scale=0.7, anchor=south] (1) at (0.5,-0.15) {$(v_0,u_0)$};
\node[draw=none,fill=none, scale=0.8, anchor=south] (1) at (1.21,-0.9) {$(v_1,u_1)$};
\node[draw=none,fill=none, scale=0.8, anchor=south] (1) at (-1.21,-0.9) {$(v_0,u_2)$};
\node[draw=none,fill=none, scale=0.7, anchor=south] (1) at (-0.5,-0.15) {$(v_1,u_0)$};
\node[draw=none,fill=none, scale=0.8, anchor=south] (1) at (1.21,0.5) {$(v_1,u_2)$};
\node[draw=none,fill=none, scale=0.8, anchor=south] (1) at (-1.21,0.5) {$(v_0,u_1)$};
	\begin{oodgraph}
		\addcycle[nodes prefix = a,color=red]{6};
		\addbeard[attach node = a->1]{2};
		\addbeard[attach node = a->2]{5};
		\addbeard[attach node = a->3]{8};
		\addbeard[attach node = a->4]{1};
		\addbeard[attach node = a->5]{8};
		\addbeard[attach node = a->6]{5};
		\addbeard[attach node = a->1->2]{3};
		\addbeard[attach node = a->3->8]{3};
		\addbeard[attach node = a->5->6]{1};
	\end{oodgraph}
\end{scope}
\end{tikzpicture}
\end{center}
    \caption{The result of the roll operation on the three unrolls in Figure \ref{fig:prod_unroll} (with $\rho$ equal to $2$, $3$, and $6$ respectively). It can thus be seen that the product $\unrollp$ of the unrolls is equivalent to the product of the FGs. }
    \label{fig:eqroll}
\end{figure}
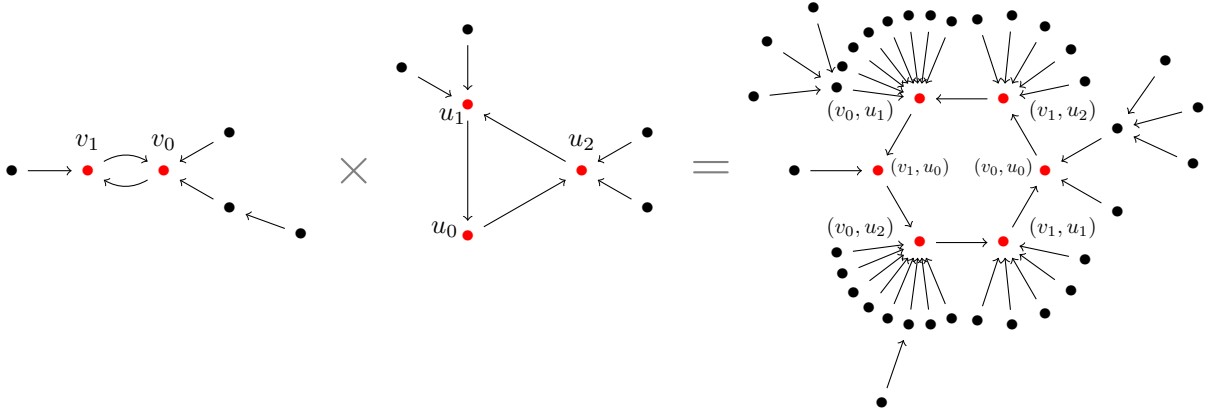

Now, if we try to fix $b'=f_B^k(b)$ and we consider $k'$ such as $f_B^{k'}(b') = f_B^{k'}(f_B^k(b)) = b$, every pairing of $b'$ with an $a\in\pp_A$ will lead, by Lemma \ref{le:isoroll}, to the same solution as fixing $f_B^{k'}(b')=b$ and $f_A^{k'}(a)$, which has already been done since we already considered every $a\in\pp_A$ with $b$. Thus, fixing one $b$ is sufficient to checking all solutions.

 
 
\begin{theorem}\label{upperbound}
For any pair of connected FGs $A,B$ and for any integer $p_X\geq1$ the inequality $A \times X\supseteq B$ admits at most $|\pp_A|$ connected solutions $X$ having $|\pp_X|=p_X$.

\end{theorem}
 
 \begin{corollary}
For any pair of connected FGs $A$ and $B$
the inequality $A \times X\supseteq B$ admits
at $|\pp_A|$ connected solutions $X$ with $|\pp_X|=p_x$ for each $p_x \ge 1$ such that $\lcm(|\pp_A|,p_x)=|\pp_B|$.

 \end{corollary}

An example of inequality where the number of solutions approaches the upper bound of Theorem~\ref{upperbound} (in the case of $\gcd(|\pp_A|, |\pp_X|)>1$) is shown in Figure~\ref{fig:ex_multiplesolutions}.
Note that this example can be generalised to create instances of the inequation $A \times X \supseteq B$ with an arbitrarily large number of solutions for $X$.

A few experiments led us to actually  conjecture that a stronger statement than Theorem~\ref{upperbound} holds when $\gcd(|\pp_A|, |\pp_X|)=1$. 
 
\begin{conjecture}
\label{conjecture}
For any pair of connected FGs $A$ and $B$, the equation $A \times X \iso B$ admits at most one solution.

\end{conjecture}

During the review process of a previous version of this work, this conjecture was actually proved, using different techniques, by Naquin and Gadouleau~\cite{naquin2022factorisation}.  
However, we plan to try to prove this statement using the same homomorphism counting techniques as Theorem~\ref{upperbound}, as this would give us new insight on the problem and possibly on related ones.
 
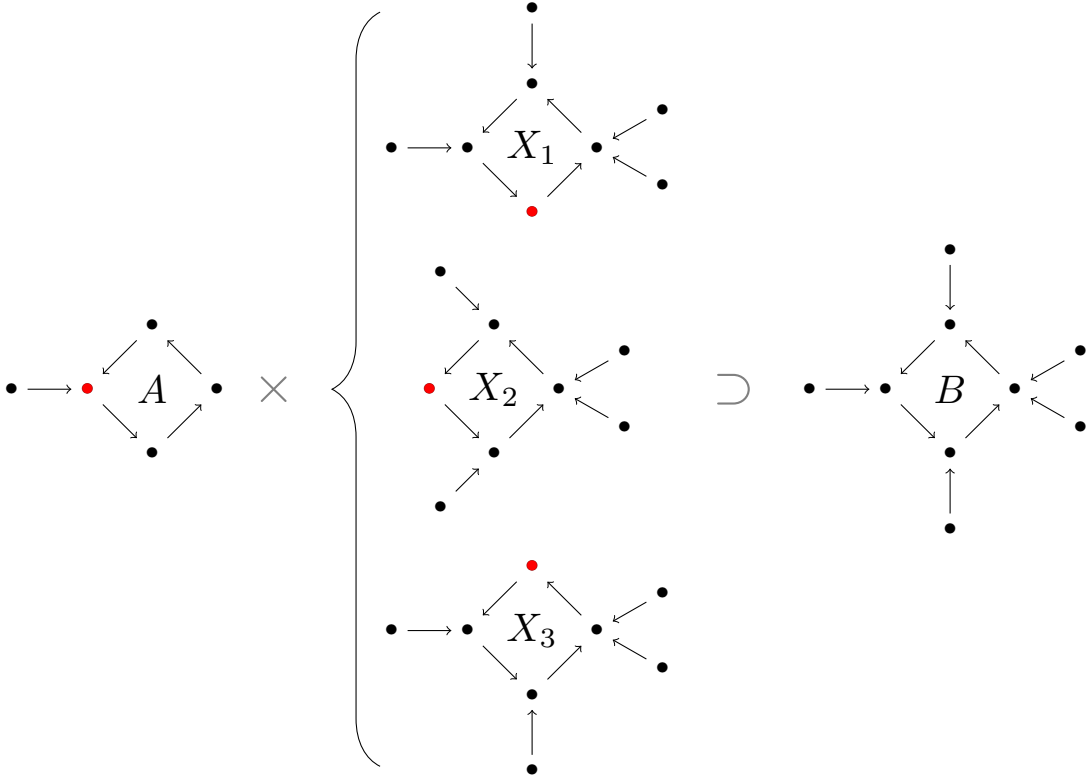
\begin{figure}
    \begin{center}
\begin{tikzpicture}
\begin{scope}[scale=1]
\begin{oodgraph}
	\addcycle[radius=.85cm,xshift=-4.5cm, nodes prefix = a]{4};
	\addcycle[radius=.85cm,xshift=0.5cm,yshift=-3.2cm, nodes prefix = x1]{4};
	\addcycle[radius=.85cm,yshift=0cm, nodes prefix = x2]{4};
	\addcycle[radius=.85cm,xshift=0.5cm,yshift=3.2cm, nodes prefix = x3]{4};
	\addcycle[radius=.85cm,xshift=6cm, nodes prefix = b]{4};
	\addbeard[attach node = a->3]{1};
	
	\addbeard[attach node = x1->1]{2};
	\addbeard[attach node = x1->3]{1};
	\addbeard[attach node = x1->4]{1};
	
	\addbeard[attach node = x2->1]{2};
	\addbeard[attach node = x2->2, rotation angle=45]{1};
	\addbeard[attach node = x2->4, rotation angle=-45]{1};
	
	\addbeard[attach node = x3->1]{2};
	\addbeard[attach node = x3->2]{1};
	\addbeard[attach node = x3->3]{1};
	
	\addbeard[attach node = b->1]{2};
	\addbeard[attach node = b->2]{1};
	\addbeard[attach node = b->3]{1};
	\addbeard[attach node = b->4]{1};
	
\end{oodgraph}
\node[draw=none,fill=none, scale=2,color=gray] at (-2.9,0) {$\times$};
\node[draw=none,fill=none, scale=2,color=gray] at (3.15,0) {$\supset$};

\node[draw=none,fill=none, scale=1.5] at (-4.5,0) {$A$};
\node[draw=none,fill=none, scale=1.5] at (0.5,3.2) {$X_1$};
\node[draw=none,fill=none, scale=1.5] at (0,0) {$X_2$};
\node[draw=none,fill=none, scale=1.5] at (0.5,-3.2) {$X_3$};
\node[draw=none,fill=none, scale=1.5] at (6,0) {$B$};

\draw [decorate,decoration={brace,amplitude=18pt}](-1.5,-5) -- (-1.5,5)  node[]{};

\node[draw=none,fill=none, color=red] at (.5,2.35) {$\bullet$};
\node[draw=none,fill=none, color=red] at (.5,-2.35) {$\bullet$};
\node[draw=none,fill=none, color=red] at (-0.85,0) {$\bullet$};
\node[draw=none,fill=none, color=red] at (-5.35,0) {$\bullet$};

\end{scope}
\end{tikzpicture}
\end{center}
    \caption{An example of an inequation with multiple solutions. The FG $A$, multiplied with one of the $X_i$, generates 4 components, one of which will be isomorphic to $B$. For one $X_i$, the cyclic alignment resulting in $B$ will be the one overlaying the cyclic node with a transient in $A$ with the cyclic node without nodes in $X_i$ (\ie, the red nodes). Note that this example can be generalised to create instances of the inequation $A \times X \supseteq B$ with an arbitrarily large number of solutions for $X$.}
    \label{fig:ex_multiplesolutions}
\end{figure}

\subsection{A polynomial algorithm for basic equations over t-abstractions}\label{poly}

In this section, we introduce a polynomial-time algorithm (Algorithm~\ref{algpoly}) to find all $\matrixt{X}{}$ (if they exist) such that $\matrixt{A}{} \times \matrixt{X}{} \supseteq \matrixt{B}{}$. Recall that according to Proposition~\ref{propprodgeneral}, the product of $\matrixt{A}{}$ and  $\matrixt{X}{}$ generates a multiset, in which we want $\matrixt{B}{}$ to be one of the elements. The algorithm takes the t-abstraction of two connected FGs $A$ and $B$ and a value $p_X$ (an admissible length of the cycle of $X$) to reconstruct inductively the possible t-abstractions of $X$. Remark that $p_X$ can be any positive natural number such that $\lcm(|\pp_A|, p_X)= |\pp_B|$. 

By Theorem~\ref{upperbound}, we have at most $|\pp_A|$ solutions (the upper bound to the number of solutions for FGs also applies for t-abstractions). Then, we can take the matrix $\matrixt{B}{}$ as it is and try to reconstruct $\matrixt{X}{}$ for each cyclic permutation of the lines of $\matrixt{A}{}$.

\medskip

The algorithm goes through every element $\matrixt{B}{}$, column by column, to compute $\matrixt{X}{}$ column by column. 
According to Proposition \ref{propprodgeneral}, we know that: $\matrixt{B}{r,h}=\matrixt{X}{r,h} \otimes M_1 + M_2$, where $M_1=\sum_{j=0}^{h-1} \matrixt{A}{r-j+i-1,h-j}$ and $M_2=\matrixt{A}{r+i-1,h} \otimes (\sum_{j=1}^{h-1} \matrixt{X}{r-j,h-j})$. 
For each $\matrixt{B}{\indexp,\tra}$, the $M_2$ is known since all $\matrixt{X}{r',h'}$ (for all $h'\in\set{1,\ldots,h-1}$ and $r'\in\set{0,\ldots,p_X-1}$) have already been computed in previous steps of the algorithm. Note that, due to how the matrix is defined, it is always possible to know the expected cardinality of a $\matrixt{X}{r,h}$.
In fact, for $h>2$, the sum of the values within $\matrixt{X}{r,h-1}$ tells us the number of nodes (thus the number of elements) in multiset $\matrixt{X}{r,h}$ (see Figure \ref{fig:tabs}).
This information is exploited in the algorithm thanks to the variable \textit{expectedCardinalities} (see Section~\ref{poly}).


At this point, the goal is to compute $\matrixt{X}{r,h}$ such that $\matrixt{B}{r,h} - M_2=\matrixt{X}{r,h} \otimes M_1$.
Finding such a $\matrixt{X}{r,h}$ can be done as follows. Let $M_3$ be the multiset $\matrixt{B}{r,h} - M_2$, and $d$, $d'$ be the maximum elements of respectively $M_1$ and $M_3$. Then, if $\matrixt{B}{}$ does indeed satisfy the product $\matrixt{A}{} \times \matrixt{X}{}$, there must exist a $y \in \matrixt{X}{r,h}$ which satisfies $d \cdot y=d'$, so we know that $\frac{d}{d'}$ is an element of $\matrixt{X}{r,h}$. We can save it and update $M_3$ to be $M_3-\multiset{\frac{d}{d'}}\otimes M_1$. We continue until either $M_3$ is empty or it contains only zeros. In this case, all the missing $y$ must be zeros too. In this last case, $\frac{|M_3|}{|M_1|}$ zeros must be added to $\matrixt{X}{r,h}$. This procedure computes all elements of $\matrixt{X}{r,h}$ in a non-ambiguous manner. We will call \emph{msDivision} the process of finding $\matrixt{X}{r,h}$ such that $M_3=\matrixt{X}{r,h} \otimes M_1$.

\medskip

Let us point out some important aspects. As for integer division, we need to consider the case where the denominator is zero (\ie, $M_1$ filled with only zeros). For such a case, one is led to think that it is impossible to define with certainty our $\matrixt{X}{r,h}$ since every multiset with the suitable cardinality would satisfy the equation. Fortunately, this particular scenario cannot happen in our case because $M_1$, as defined above, contains the multiset $\matrixt{A}{r-h+1,1}$ which necessarily contains a positive integer.

It may happen, when we go through the $\matrixt{B}{\indexp,\tra}$
with $\indexp > p_X-1$, that the corresponding $\matrixt{X}{\indexp,\tra}$ has already been computed. 
This is normal, since the rows of $\matrixt{X}{}$ are considered modulo $p_X$.
In this case, we only check if the product holds (an example of this mechanism is shown below).

To conclude, if a contradiction occurs at any point of the algorithm, the reconstruction of the current solution is stopped, and the algorithm tries out another cyclic permutation of the lines of $\matrixt{A}{}$ (\ie, another $i$). Such contradictions occur, for instance, when we compute a difference between two multisets and the second one is not included in the first one, when the computed value $\frac{d}{d'}$ is not an integer, or when the cardinality of the multiset returned by $msDivision$ does not match the expected one. 

\medskip

An implementation of the approach just outlined is presented in Algorithm \ref{algpoly}.
\subsubsection{A worked-out example}
\label{ex:poly}

In this example we go over the details of Algorithm~\ref{algpoly} and we will explain how to reconstruct the $\matrixt{X}{}$ up to its second column (\ie, the two red ones) with the data provided below.

\definecolor{brick}{RGB}{170,7,7}
\begin{figure}[h!]
\centering
\begin{tikzpicture}

\node[scale=1.2]() at (-7,1.5){$\matrixt{A}{}=$};
\node[scale=1.2]() at (-.75+1,1.5){$\matrixt{X}{}=$};
\node[scale=1.2]() at (-7,-1.5){$\matrixt{B}{}=$};

\node (0) at (-4.5,1.5) {
\resizebox{!}{!}{
$\begin{blockarray}{cccc}
 & 1 & 2 & 3  \\
\begin{block}{c(ccc)}
  0 & [2] &       [0] &  \emptyset  \\
  1 & [5] & [0,0,0,1] & [0] \\
\end{block}
\end{blockarray}$}};

\node (1) at (3+1,1.5) {
\resizebox{!}{!}{
$\begin{blockarray}{ccccccc}
 &              1 &        2 &              3 &        4 &              5 &       6  \\
\begin{block}{c(cccccc)}
0 & \textcolor{brick}{[2]} &     \textcolor{brick}{[2]} & \textcolor{gray}{[0,1]} & \textcolor{gray}{[2]} & \textcolor{gray}{[0,1]} & \textcolor{gray}{[0]} \\
1 & \textcolor{brick}{[4]} & \textcolor{brick}{[0,0,0]} &      \textcolor{gray}{\emptyset} &    \textcolor{gray}{\emptyset} &      \textcolor{gray}{\emptyset} &    \textcolor{gray}{\emptyset} \\
2 & \textcolor{brick}{[3]} &   \textcolor{brick}{[0,1]} &   \textcolor{gray}{[0]} &    \textcolor{gray}{\emptyset} &      \textcolor{gray}{\emptyset} &    \textcolor{gray}{\emptyset}\\
\end{block}
\end{blockarray}$}};

\node (2) at (0.5,-1.5) {
\resizebox{.85\textwidth}{!}{
$\begin{blockarray}{ccccccc}
  &             1 &                                                2 &                            3 &                    4 &                          5 &                    6 \\
\begin{block}{c(cccccc)}
0 &  \multiset{4} &                                \multiset{0,0,10} &                 \multiset{2} &         \multiset{1} & \multiset{0,0,0,0,0,0,0,0} &       \multiset{0,0} \\
1 & \multiset{20} & \multiset{0,0,0,0,0,0,0,0,0,0,0,0,0,0,0,0,0,0,2} &               \multiset{0,0} &            \emptyset &                  \emptyset &            \emptyset \\
2 &  \multiset{6} &                             \multiset{0,0,0,0,5} &         \multiset{0,0,0,0,0} &            \emptyset &                  \emptyset &            \emptyset \\
3 & \multiset{10} &                     \multiset{0,0,0,0,0,0,2,3,4} & \multiset{0,0,0,0,0,0,0,0,5} & \multiset{0,0,0,2,4} &     \multiset{0,0,0,0,0,5} & \multiset{0,0,0,0,0} \\
4 &  \multiset{8} &                         \multiset{0,0,0,0,0,0,0} &                    \emptyset &            \emptyset &                  \emptyset &            \emptyset \\
5 & \multiset{15} &           \multiset{0,0,0,0,0,0,0,0,0,0,0,1,2,4} &     \multiset{0,0,0,0,0,0,0} &            \emptyset &                  \emptyset &            \emptyset \\
\end{block}
\end{blockarray}$}};

\end{tikzpicture}
\end{figure}

At the beginning of the algorithm the only information known about $\matrixt{X}{}$ is that $p_X$ equals $3$. We thus start by initialising a matrix with $3$ lines and $6$ columns (the number of columns of $\matrixt{B}{}$) filled with empty multisets. The \textit{expectedCardinalities} array, which stores at any time the mandatory cardinalities of the multisets at the next layer, is filled with $p_X$ ones. We will start by considering $i=1$.

For the first layer (\ie, $\tra=1$), let us consider the computation made by the algorithm for $r$ up to $2$:
\begin{align*}
    r=0 & \rightarrow M_2 = \matrixt{A}{0,1} \otimes \emptyset = \multiset{2}\otimes\emptyset , \matrixt{X}{0,1}=msDivision(\matrixt{B}{0,1}-\emptyset,\matrixt{A}{0,1})=msDivision(\multiset{4},\multiset{2})=\multiset{2}\\
    r=1 & \rightarrow M_2 = \multiset{5} \otimes \emptyset ,\matrixt{X}{1,1}=msDivision(\multiset{20},\multiset{5})=\multiset{4}\\
    r=2 & \rightarrow M_2 = \multiset{2} \otimes \emptyset ,\matrixt{X}{2,1}=msDivision(\multiset{6},\multiset{2})=\multiset{3}
\end{align*}

Since we are on the first layer, \textit{expectedCardinalities} is updated to be the sum of the computed multisets minus $1$ (since the cyclic preimage will be not considered). We will thus have $[1,3,2]$.
 
From $r=3$, the algorithm will now enter the verification phase and we will have:

\begin{align*}
    r=3 & \rightarrow \matrixt{A}{1,1} \otimes \matrixt{X}{0,1} = \matrixt{B}{3,1} , \multiset{5}\otimes\multiset{2} = \multiset{10}\\
    r=4 & \rightarrow \matrixt{A}{0,1} \otimes \matrixt{X}{1,1} = \matrixt{B}{4,1} , \multiset{2}\otimes\multiset{4} = \multiset{8}\\
    r=5 & \rightarrow \matrixt{A}{1,1} \otimes \matrixt{X}{2,1} = \matrixt{B}{5,1} , \multiset{5}\otimes\multiset{3} = \multiset{5}
\end{align*}

For the second layer, the algorithm proceeds in the same way:

\begin{align*}
    r=0 & \rightarrow M_2 = \matrixt{A}{0,2} \otimes \matrixt{X}{2,1} = \multiset{0}\otimes\multiset{3} ,\matrixt{X}{0,2}=msDivision(\matrixt{B}{0,2}-\multiset{0},\matrixt{A}{1,1})=msDivision(\multiset{0,10},\multiset{0,5})=\multiset{2}\\
    r=1 & \rightarrow M_2 = \multiset{0,0,0,1} \otimes \multiset{2} ,\matrixt{X}{1,2}=msDivision(\multiset{0,0,0,0,0,0,0,0,0,0,0,0,0,0,0},\multiset{0,0,0,1,2})=\multiset{0,0,0}\\
    r=2 & \rightarrow M_2 = \multiset{0} \otimes \multiset{4} ,\matrixt{X}{2,2}=msDivision(\multiset{0,0,0,5},\multiset{0,5})=\multiset{0,1}
\end{align*}

No error is triggered since the cardinalities of the computed multisets match the expected ones. The \textit{excpectedCardinalites} array is then updated to be $[2,0,1]$. We can again check if these multisets work for the remaining lines.

\begin{align*}
r=3 \rightarrow& \matrixt{A}{1,2} \otimes \matrixt{X}{0,2} + \matrixt{A}{0,1} \otimes \matrixt{X}{0,2} + \matrixt{A}{1,2} \otimes \matrixt{X}{2,1} = \matrixt{B}{3,2}\\
&\multiset{0,0,0,1}\otimes\multiset{2}+\multiset{2}\otimes\multiset{2}+\multiset{0,0,0,1}\otimes\multiset{3} = \multiset{0,0,0,0,0,0,2,3,4}\\
r=4 \rightarrow& \matrixt{A}{0,2} \otimes \matrixt{X}{1,2} + \matrixt{A}{1,1} \otimes \matrixt{X}{1,2} + \matrixt{A}{0,2} \otimes \matrixt{X}{0,1} = \matrixt{B}{4,2}\\
&\multiset{0}\otimes\multiset{0,0,0}+\multiset{5}\otimes\multiset{0,0,0}+\multiset{0}\otimes\multiset{2} = \multiset{0,0,0,0,0,0,0}\\
r=5 \rightarrow& \matrixt{A}{1,2} \otimes \matrixt{X}{2,2} + \matrixt{A}{0,1} \otimes \matrixt{X}{2,2} + \matrixt{A}{1,2} \otimes \matrixt{X}{1,1} = \matrixt{B}{5,2}\\
&\multiset{0,0,0,1}\otimes\multiset{0,1}+\multiset{2}\otimes\multiset{0,1}+\multiset{0,0,0,1}\otimes\multiset{4} = \multiset{0,0,0,0,0,0,0,0,0,0,0,1,2,4}
\end{align*}

Since no contradiction is found, the algorithm can go on.

\begin{algorithm}[!ht]
\caption{Polynomial t-abstraction}
\label{algpoly}
\SetKwInOut{Input}{Input}\SetKwInOut{Output}{Output}
\Input{$\matrixt{A}{}$, $\matrixt{B}{}$ matrices of integer multisets, and $p_X$ positive natural number}
\Output{A list of solutions for $\matrixt{X}{}$}
 $solutions \newv{\gets} list()$\;
 \ForAll{i $\in \set{1, \ldots, \mid \cycle_A \mid}$}{
        $h\newv{\gets}1$\;
        $complete\newv{\gets}false$\;
        $error\newv{\gets}false$\;
        $expectedCardinalities\newv{\gets} [1] * p_X$\;
        $\matrixt{X}{} \newv{\gets} emptyMatrix(p_X,h_B^{\max})$\;
        \While{$complete=false \text{ \textbf{and} } error=false$}{ 
        $complete\newv{\gets}true$\;
            \ForAll{r $\in \set{0, \ldots, \mid \cycle_B \mid-1}$}{
                \eIf{$r< p_X$}{
                    \If{$expectedCardinalities[r]=0$}{
                        \textbf{continue}\;
                    }
                    $M_2 \newv{\gets} \matrixt{A}{r+i-1,h} \otimes (\sum_{j=1}^{h-1} \matrixt{X}{r-j,h-j})$\;
                    $error,\matrixt{X}{r,h}\newv{\gets}msDivision(\matrixt{B}{r,h} - M_2, \sum_{j=0}^{h-1} \matrixt{A}{r-j+i-1,h-j})$\;
                    $nextCardinality \newv{\gets} sum(\matrixt{X}{r,h})$\;
                    $error \newv{\gets}error \lor (\mid \matrixt{X}{r,h} \mid \neq expectedCardinalities[r])$\;
                    \eIf{$h=1$}{
                        $expectedCardinalities[r] \newv{\gets} nextCardinality -1$\;
                    }{
                        $expectedCardinalities[r] \newv{\gets} nextCardinality$\;
                    }
                    
                    \If{$nextCardinality>0$}{
                         $complete\newv{\gets}false$\;
                    }
                }{
                    $R\newv{\gets}\matrixt{A}{r,h}\otimes(\sum_{j=0}^{h-1} \matrixt{X}{r-j+(i-1),h-j})+ \matrixt{X}{r+(i-1),h}\otimes(\sum_{j=1}^{h-1} \matrixt{A}{r-j,h-j})$\;
                    $error\newv{\gets}(\matrixt{B}{r,h}\neq R)$\;\label{algpoly:check}
                }
                \If{$error$}{
                    \textbf{break}\;
                }
            }
            $h\newv{\gets}h+1$\;
        }
        \If{$!error$}{
            $solutions.append(\matrixt{X}{})$\;
        }
 }
 \Return{solutions}\;
\end{algorithm}

\subsubsection{Complexity}
Each update of $M_3$ in $msDivision$ requires linear time \wrt the size of $M_3$. We repeat it for each element of $\matrixt{X}{r,h}$, namely, $\frac{|M_3|}{|M_1|}$ times, with $|M_1|$ possibly equal to $1$. Therefore, \emph{msDivision} complexity is in $O({|M_3|}^2)$.
Due to the definition of the t-abstraction, we know that $\matrixt{B}{}$ contains $| \setst_B |$ integer values partitioned into the different $\matrixt{B}{\indexp,\tra}$. Since the $msDivision$ is applied to all $M_3$ (or $\matrixt{B}{\indexp,\tra} - M_2$ with $M_2$ potentially empty), the worst case is when $msDivision$ is applied to one $M_3$ containing $|\setst_B|$ elements. If we consider also that we go through the matrix, for a certain $i$, the complexity is in  $O(|\pp_B|\cdot\tra_B^{\max}+{|\setst_B|}^2)$.
Since we evaluate $i$ from $0$ to $|\pp_A|-1$, the total complexity of the algorithm is $O(| \pp_A | \cdot (|\pp_B|\cdot\tra_B^{\max}+{| \setst_B |}^2))$. Note that this algorithm does not ensure that the corresponding equation over FGs has a solution. However, it ensures that if the latter has solutions, they satisfy one of the t-abstractions found by the algorithm. 
Nevertheless, if a valid t-abstraction is not found for $X$, one can conclude that the original Equation~\eqref{DDSineq} is impossible.

\subsection{An exponential algorithm for basic equations}\label{exp}

This section introduces an exponential-time algorithm (Algorithm~\ref{algexpG}) which takes
as input the FGs $A$ and $B$ and an additional integer $p_X$,
and outputs all $X$ which satisfy $A \times X \supseteq B$.
The polynomial-time algorithm of the previous section can find 
all $\matrixt{X}{}$ such that $\matrixt{A}{} \times \matrixt{X}{} \supseteq \matrixt{B}{}$, if they exist. Unfortunately, these solutions lack information to completely reconstruct the dynamics of $X$. Since the multisets are not ordered, we cannot know which indegree of a $\matrixt{X}{\indexp,\tra}$ is related to which node in $\tpmatrix{X}{\indexp,\tra-1}$. Figure~\ref{fig:2dds_1abs} illustrates this potential ambiguity. Each solution $\matrixt{X}{}$ identifies an equivalence class of FGs respecting the abstraction. This version of the algorithm is an evolution of the polynomial one and leads us to identify which of these FGs are solutions of the original equation $A \times X \supseteq B$.

\subsubsection{Naive reconstruction}
The most intuitive way to reconstruct the graph $X$, starting from $\matrixt{X}{}$ (computed by Algorithm~\ref{algpoly}), is to test all the possible connected FGs satisfying the t-abstraction. This corresponds to testing all the ways to connect the elements of $\matrixt{X}{\indexp,\tra}$ to the elements of $\tpmatrix{X}{\indexp,\tra-1}$ (level by level).
Since $|\matrixt{X}{\indexp,\tra}|=|\tpmatrix{X}{\indexp,\tra-1}|$, for a level $\tra$ and an index $\indexp$ of a node on the cycle, we have $|\matrixt{X}{\indexp,\tra}|!$ possibilities. Then, there are $\prod_{(\indexp,\tra)\in\set{0,\ldots,p_X-1}\times\set{1,\ldots,h^{\max}_X}}^{} (|\matrixt{X}{\indexp,\tra}|!) $ possible systems (not up to isomorphism).
However, we know that the polynomial algorithm returns at most one t-abstraction solution for each cyclic permutation of $\matrixt{A}{}$ (\ie, one for each possible cyclic alignment of $A$ on $B$). Then, according to Theorem~\ref{upperbound}, for each $\matrixt{X}{}$ returned by the polynomial method, there is just one possible graph that is a solution of the basic equation over FGs. To verify if a graph is a solution, one can check if, in the graph obtained from the product of itself with $A$, the component corresponding to the alignment considered before is isomorphic to $B$.

\subsubsection{An improved approach}
Knowing the full dynamics of $A$ and $B$ allows us to associate each integer in a multiset (\ie, each incoming degree) of $\matrixt{A}{}$ or $\matrixt{B}{}$ with the corresponding node in the graph. In practice, this can be done by providing the multisets with an order such that the $k$-th node in a $\tpmatrix{A}{\indexp,\tra}$ has the corresponding $k$-th indegree in $\matrixt{A}{\indexp,\tra+1}$ (the same applies for $B$). In the same way as the polynomial algorithm, this version will go through every layer $\tra$ of the transients to reconstruct $X$ layer by layer.

At first, $X$ consists only of a cycle of length $p_X$.
At level $\tra=1$, the multisets computed as in the polynomial algorithm are sufficient to reconstruct the dynamics of the first layer of $X$. Indeed, every $\matrixt{X}{\indexp,1}$ has only one element $d$. Then, we just attach, to the $r$-th cyclic node, $d$ new nodes.
At $\tra=2$, the multisets derived from the polynomial algorithm are still sufficient to reconstruct the dynamics because all nodes in $\tpmatrix{X}{\indexp,1}$ are equivalent from a dynamics point of view. Then, for each $d\in\matrixt{X}{\indexp,2}$, we assign $d$ predecessors to an arbitrary node in $\tpmatrix{X}{\indexp,1}$ (which must still be a leaf).
Now, the nodes in $\tpmatrix{X}{\indexp,2}$ are potentially no longer equivalent. 
For this reason, starting from $\tra=3$, the algorithm will make assumptions on which node of $X$ at the layer $\tra$ has generated which node of $B$ at the same layer. These assumptions allow us to reconstruct the next layer of $X$ without ambiguity. We will call these assumptions the \emph{origins} of the nodes of $B$. From a graph point of view, the origins will be only labels on nodes of $X$ and $B$ such that all the nodes in $B$ with origin $\ori$ are supposed to be generated by the node with label $\ori$ in $X$. Let $\oris_{\indexp,\tra}$ be the set of origins of the nodes in a certain  $\tpmatrix{B}{\indexp,\tra}$.

\medskip

We now consider the computation of a generic layer $\tra\geq3$. Let us suppose that we know the origins of the nodes in $\tpmatrix{B}{\indexp,\tra-2}$ (we will see later how origins are initialised). With this information, we can partition the multiset $\matrixt{B}{\indexp,\tra}$ into sub-multisets depending on the origin of the nodes in $\tpmatrix{B}{\indexp,\tra-2}$. Let $\tpmatrix{B}{\indexp,\tra|\ori}$ be the set of all $\nodedds\in\tpmatrix{B}{\indexp,\tra}$ such that $f_B(\nodedds)$ has origin $\ori$.  Then, let $\matrixt{B}{\indexp,\tra|\ori}$ be the multiset of indegrees of nodes $\nodedds\in\tpmatrix{B}{\indexp,\tra-1|\ori}$. According to these definitions, $\matrixt{B}{\indexp,\tra}=\sum_{\ori\in\oris_{r,h-2}}\matrixt{B}{r,h|\ori}$. We can also partition the multiset of unknowns $\matrixt{X}{\indexp,\tra}$, as $\matrixt{X}{\indexp,\tra}=\sum_{\ori\in\oris_{r,h-2}}\matrixt{X}{r,h|\ori}$. Consequently, we can rewrite the equivalence of Proposition \ref{propprodgeneral} as follows:
\begin{equation}\label{epprodorigins}
    \sum_{\ori\in\oris_{r,h-2}}\matrixt{B}{r,h|\ori}=\sum_{\ori\in\oris_{r,h-2}}\left(\matrixt{X}{r,h|\ori}\otimes\sum_{j=0}^{h-1}\matrixt{A}{r-j,h-j}\right)+\matrixt{A}{r,h}\otimes\sum_{j=1}^{h-1} \matrixt{X}{r-j,h-j}.
\end{equation}

Note that there are some nodes in $\tpmatrix{B}{\indexp,\tra-1}$ which are not generated from a node in $\tpmatrix{X}{\indexp,\tra-1}$. These are the nodes of a level of $B$ that come from the multiplication of lower levels of $X$ with the last level of $A$.
The indegrees of these nodes are the ones found in the second term (\ie, $\matrixt{A}{r,h}\otimes\sum_{j=1}^{h-1} \matrixt{X}{r-j,h-j}$), which we denote $M_2$ as in the previous section.
These nodes are marked with a special origin (noted $-1$).  
Then, we should always have $\matrixt{B}{\indexp,\tra|-1}=M_2$ (line \ref{algexp:checkMinusOne} Algorithm \ref{algexp}). Let us point out that, according to this reasoning, we indeed have that all the nodes attached to a node with $-1$ origin will also have $-1$ as origin. 

Considering the definition of $M_1$ (presented in Section \ref{poly}) and Equation \eqref{epprodorigins}, we can now decompose the product formula into smaller ones.
We thus have $\matrixt{B}{\indexp,\tra|\ori}=\matrixt{X}{\indexp,\tra|\ori} \otimes M_1$ for all $\ori\in\oris_{\indexp,\tra-2}\setminus\set{-1}$. At this point, we can compute each $\matrixt{X}{\indexp,\tra|\ori}$ with $msDivision$. 
Once $\matrixt{X}{\indexp,\tra|\ori}$ is known, we have to assign to each node of $\tpmatrix{X}{\indexp,\tra-1|\ori}$ one degree (\ie, a set of predecessors) from $\matrixt{X}{\indexp,\tra|\ori}$. However, we can notice that we brought the problem back to the special case of $\tra=2$. Indeed, each node of $\tpmatrix{X}{\indexp,\tra-1|\ori}$ is connected to the same node (the only one with origin $\ori$). This means that we can again, for each $d\in\matrixt{X}{\indexp,\tra|\ori}$, attach $d$ new nodes to an arbitrary node of $\tpmatrix{X}{\indexp,\tra-1|\ori}$ (line \ref{algexp:attachGeneralCase} Algorithm \ref{algexp}). This mechanism is the key to reconstructing $X$ layer by layer without ambiguity.

Incidentally, each time we add a $k$-th new node in $X$, we set its label equal to $k$. In Algorithm \ref{algexp}, this is handled by the $attachNewNodes$ method which, in addition to creating the nodes, assigns to them the corresponding origin\footnote{The $cyclicGraph$ function Algorithm \ref{algexp} automatically assigns origins while creating the cyclic nodes of $X$.}. Moreover, it returns $true$ if at least one node is created. Hence, this ensures that all nodes in $X$ have a unique label. The returned value, for its part, allows us to know when no more new nodes are added in $X$ and thus when $X$ has been completely reconstructed.

\medskip

There still remains an issue: how to assign origins to the nodes in the layer of $B$. Up to now, we assigned to each node in $\tpmatrix{X}{\indexp,\tra-1}$ a number of predecessors. This means that we also know for each node $v\in\tpmatrix{X}{\indexp,\tra-1|\ori}$ all the correlated degrees in $\matrixt{B}{\indexp,\tra|\ori}$. Indeed, if $v$ has an incoming degree $d$, there is a sub-multiset $M_v=\multiset{d}\otimes M_1$ included in $\matrixt{B}{\indexp,\tra|\ori}$.
We must then apply the origin of $v$ as the origin of all corresponding nodes in $\tpmatrix{B}{\indexp,\tra-1|\ori}$. 
In other words, for each element $d'$ in $M_1$, we assign the origin of $v$ to any node $w$ in $\tpmatrix{B}{\indexp,\tra-1|\ori}$ that has $d\cdot d'$ predecessors.
However, several choices may be possible.
In fact, two nodes in $\tpmatrix{B}{\indexp,\tra-1|\ori}$ can have the same number of predecessors, and both can be valid choices at this point\footnote{The $enumerateAssignments$ function takes as parameters a list of constraints over origins. The constraints are triplets made of (\emph{i}) the name of the origin, (\emph{ii}) the multiset $M_v$ of indegrees on which the origins must be assigned, and (\emph{iii}) a set of nodes among which we can assign it. It then outputs all the valid assignments (hashmaps linking nodes to origins).}. 
A possible approach is backtracking. Namely, we choose one and we try to reconstruct the layers above (lines \ref{algexp:enumerateAssignments} and~\ref{algexp:recursiveCall} of Algorithm \ref{algexp}). 
If the reconstruction of an above layer fails, the algorithm backtracks to the last choice made and tries another origin assignment (lines \ref{algexp:backtrack1} and \ref{algexp:backtrack2}). The enumeration of all valid choices for each layer is the cause of the exponential time complexity. In a following section, we will present a better approach to deal with this problem.

In any case, at this point, we have an almost complete technique for assigning origins to nodes. What remains to be understood is how the special origin $-1$ is initialised.
The first layer where the $M_2$ term appears is at height $\tra=2$. The origins are not needed here to construct $X$ but must be assigned for the layer $\tra=3$. 
We proceed as explained for $h=2$, and then we only assign a $-1$ origin to all the nodes in $\tpmatrix{B}{\indexp,1}$ remaining without origins.

\medskip

For the verification part, at each height $h$ after having reconstructed the layer of $X$, we check that $A \times X$ (considered only up to layer $h$) contains $B$ (again only up to height $h$). The \emph{truncation} of the graphs, consisting of returning subgraphs by removing transient nodes in higher levels, is done by the function $truncate$ in Algorithm~\ref{algexp}.

\medskip 

As for the polynomial algorithm, we actually repeat this whole mechanism for all cyclic alignments (Algorithm~\ref{algexpG}). Let us point out that according to the results of Section \ref{cancellation}, after finding a solution for a cyclic alignment (\ie, for a value $i$), the algorithm starts calculating a new solution for another alignment (even if not all possible origin assignments were explored via backtrack).

\begin{algorithm}[!ht]
\caption{Exponential equation solver}
\label{algexpG}
\SetKwInOut{Input}{Input}\SetKwInOut{Output}{Output}
\Input{$A$, $B$ connected functional graphs, and $p_X$ positive natural number}
\Output{A list of solutions for $X$}
 $solutions \newv{\gets} list()$\;
 \ForAll{i $\in \set{1, \ldots, \mid \cycle_A \mid}$}{
    $X\newv{\gets}cyclicGraph(p_X)$ \Comment*[r]{\small $cyclicGraph$ also assign new origins on $X$}
    \If{$findSolution(i,A,B,X,1)$}{
        $solutions.append(clone(X))$ \Comment*[r]{\small A copy of $X$ is saved as a solution}
    }
}
 \Return{solutions}\;
\end{algorithm}

\begin{algorithm}[!ht]
\caption{findSolution}
\label{algexp}
\SetKwInOut{Input}{Input}\SetKwInOut{Output}{Output}
\Input{$i$ positive natural number, $A$, $B$, $X$ connected functional graphs and $h$ positive natural number}
\Output{$true$ if a solution has been reconstructed, $false$ otherwise\\($X$ is modified to be the solution, if it exists, and is not returned)}
$originsData \newv{\gets} list()$ \Comment*[r]{\small list giving, for a origin, the corresponding nodes in $X$ and $B$}
$nodeAttached \newv{\gets} false$\;
\ForAll{$r \in \set{1,\ldots,|\cycle_X|}$}{
    \If{$h=1$}{
        $error,\matrixt{X}{r,1} \newv{\gets} msDivision(\matrixt{B}{r,1},\matrixt{A}{r+i-1,1})$\;
        $nodeAttached \newv{\gets} nodeAttached \lor attachNewNodes(g_X(r),\matrixt{X}{r,1}[0]-1)$\;
        $originsData.append\big( \,(getOrigin(g_X(r)),\matrixt{B}{r,1},\multiset{g_B(r)}) \, \big)$\;
    }
    \If{$h=2$}{
        $M_1 \newv{\gets} \matrixt{A}{r+i-2,1}+\matrixt{A}{r+i-1,2}$\;
        $error,\matrixt{X}{r,2} \newv{\gets} msDivision(\matrixt{B}{r,2}-(\matrixt{A}{r+i-1,2}\otimes\matrixt{X}{r-1,1}),M_1)$\;
        \ForAll{$k \in \set{0,\ldots,|\matrixt{X}{r,2}|-1}$}{
            $nodeAttached \newv{\gets} nodeAttached \lor attachNewNodes(\tpmatrix{X}{r,1}[k],\matrixt{X}{r,2}[k])$\;
            $originsData.append\big( \, (getOrigin(\tpmatrix{X}{r,1}[k]),\matrixt{X}{r,2}[k] \otimes M_1,\tpmatrix{B}{r,1}) \, \big)$\;
        }
        $originsData.append\big( \, (-1,\matrixt{A}{r+i-1,2} \otimes \matrixt{X}{r-1,1},\tpmatrix{B}{r,1}) \, \big)$\;
    }
    \If{$h>2$}{
        $M_1 \newv{\gets} \sum_{j=0}^{h-1} \matrixt{A}{r-j+i-1,h-j}$\;
        $M_2 \newv{\gets} \matrixt{A}{r+i-1,h} \otimes (\sum_{j=1}^{h-1} \matrixt{X}{r-j,h-j})$\;
        \ForAll{$\ori \in \oris_{r,\tra-2}$}{
            \eIf{$\ori=-1$}{
                $error \newv{\gets} (M_2 \neq \matrixt{B}{r,h|-1})$\; \label{algexp:checkMinusOne}
                $originsData.append\big( \, (-1,M_2,\tpmatrix{B}{r,h-1|-1}) \, \big)$ \Comment*[r]{\small $-1$ origins automatically propagate}
            }{
                $error,\matrixt{X}{r,h|\ori}\newv{\gets}msDivision(\matrixt{B}{r,h|\ori}, M_1)$\;
                \ForAll{$k \in \set{0,\ldots,|\matrixt{X}{r,h|\ori}|-1}$}{
                    $nodeAttached \newv{\gets} nodeAttached \lor attachNewNodes(\tpmatrix{X}{r,h-1|\ori}[k],\matrixt{X}{r,h|\ori}[k])$\; \label{algexp:attachGeneralCase}
                    $originsData.append\big( \, (getOrigin(\tpmatrix{X}{r,h-1|\ori}[k]),\matrixt{X}{r,h|\ori}[k] \otimes M_1,\tpmatrix{B}{r,h-1|\ori})\, \big)$\;
                }
            }
        }
    }
}
\If{$error \text{ \textbf{or} } \lnot (truncate(A,h) \times X \supseteq truncate(B,h))$}{
    \Return{false}\;
}
\If{$\lnot nodeAttached \text{ \textbf{and} } (A \times X \supseteq B)$}{
    \Return{true}\;
}
\ForAll{$origins \in enumerateAssignements(originsData)$}{ \label{algexp:enumerateAssignments}
    $applyOrigins(B,origins)$\;
    \If{$findSolution(i,A,B,X,h+1)$}{ \label{algexp:recursiveCall}
        \Return{true};
    }
    $detachOrigins(B,origins)$\; \label{algexp:backtrack1}
}
$X \newv{\gets} truncate(X,h-1)$\; \label{algexp:backtrack2}
\Return{false}\;
\end{algorithm}

\subsubsection{Another worked-out example}

\newcommand{\squeezeB}{0.5}
\newcommand{\squeezeX}{0.6}

\tikzset{
    tnode/.style={
        draw=none,
        fill=none,
        scale=1
    },
    every loop/.style={
        looseness=5
    }
}

\begin{figure}[ht!]
\begin{center}
\begin{tikzpicture}
\begin{scope}[shift={(3,0)}]
    \foreach \n in {0,...,11}{
        \node[tnode] (l3n\n) at (\squeezeB*\n,3) {$\bullet$};
    }
    \foreach \n [count=\ni from 0] in {1.5,4.5,6.5,8,9.5,11,12,13,14,15,16,17}{
        \node[tnode,scale=.9] (l2n\ni) at (\squeezeB*\n,2) {$v_{\ni}$};
    }
    \foreach \n [count=\ni from 0] in {4.75,10.25,13,15.5,17}{
        \node[tnode] (l1n\ni) at (\squeezeB*\n,1) {$\bullet$};
    }
	\node[tnode] (l0n0) at (\squeezeB*10.5,0) {$\bullet$};
	
	\foreach \n [count=\ni from 0] in {0,0,0,0,1,1,2,2,3,4,4,5}{
	    \draw[->](l3n\ni) to (l2n\n);
	}
	\foreach \n [count=\ni from 0] in {0,0,0,0,1,1,2,2,2,3,3,4}{
	    \draw[->](l2n\ni) to (l1n\n);
	}
	\foreach \n in {0,...,4}{
	    \draw[->](l1n\n) to (l0n0);
	}
	\path[->](l0n0) edge [out=225, in=315, loop] (l0n0);
	
	
	\node[tnode,scale=0.8]at( 4.75*\squeezeB-.3,1-.1){$(1)$};
	\node[tnode,scale=0.8]at(10.25*\squeezeB-.3,1){$(2)$};
	\node[tnode,scale=0.8]at(   13*\squeezeB-.4,1){$(-1)$};
	\node[tnode,scale=0.8]at(15.50*\squeezeB-.3,1+.1){$(1)$};
	\node[tnode,scale=0.8]at(   17*\squeezeB+.3,1){$(2)$};
	
	\node[tnode,scale=0.8]at(\squeezeB*10.5-.45,-.1){$(0)$};

	\node[tnode]at(\squeezeB*10.5,-1){$B$};
  
\end{scope}
\begin{scope}[shift={(0,0)}]
    \foreach \n in {0,1,2,3}{
        \node[tnode,color=red] (l3n\n) at (\squeezeX*\n,3) {$\bullet$};
    }
    \foreach \n in {0,1,2}{
        \node[tnode] (l2n\n) at (\squeezeX*\n+\squeezeX*0.5,2) {$\bullet$};
    }
    \foreach \n in {0,1}{
        \node[tnode] (l1n\n) at (\squeezeX*\n+\squeezeX*1,1) {$\bullet$};
    }
	\node[tnode] (l0n0) at (\squeezeX*1.5,0) {$\bullet$};
	\foreach \n [count=\ni from 0] in {0,0,1,2}{
	    \draw[->,color=red](l3n\ni) to (l2n\n);
	}
	\foreach \n [count=\ni from 0] in {0,0,1}{
	    \draw[->](l2n\ni) to (l1n\n);
	}
	\foreach \n in {0,1}{
	    \draw[->](l1n\n) to (l0n0);
	}
	\path[->](l0n0) edge [out=225, in=315, loop] (l0n0);
	
	\node[tnode,scale=0.8]at(\squeezeX*1.5-.3,0){$(0)$};
	\node[tnode,scale=0.8]at(\squeezeX*1  -.3,1){$(1)$};
	\node[tnode,scale=0.8]at(\squeezeX*2  +.3,1){$(2)$};
	\node[tnode,scale=0.8]at(\squeezeX*0.5-.3,2){$(3)$};
	\node[tnode,scale=0.8]at(\squeezeX*1.5-.3,2){$(4)$};
	\node[tnode,scale=0.8]at(\squeezeX*2.5+.3,2){$(5)$};

	\node[tnode]at(\squeezeX*1.5,-1){$X$};
\end{scope}
\begin{scope}[shift={(-1,0)}]
    \node[tnode] (l2n0) at (0,2) {$\bullet$};
    \node[tnode] (l1n0) at (0,1) {$\bullet$};
	\node[tnode] (l0n0) at (0,0) {$\bullet$};
	\draw[->](l2n0) to (l1n0);
	\draw[->](l1n0) to (l0n0);
	\path[->](l0n0) edge [out=225, in=315, loop] (l0n0);
	
	\node[tnode]at(0,-1){$A$};
\end{scope}
\end{tikzpicture}
\end{center}
    \caption{Three functional graphs $A$, $X$ and $B$ (origins are shown in parentheses).}
    \label{fig:ex_algexp}
\end{figure}
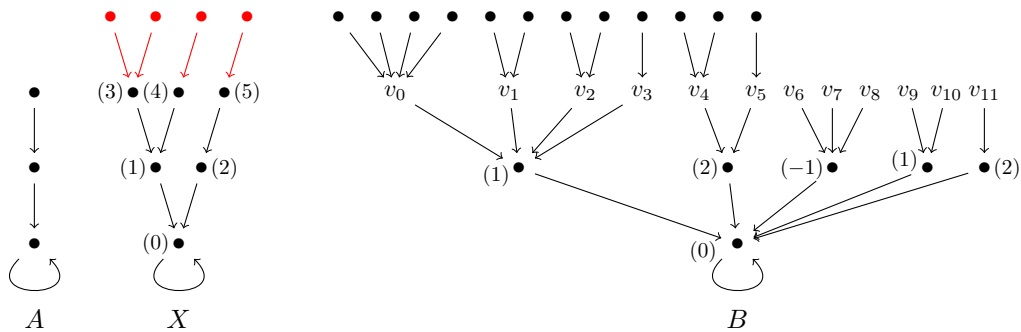

We will show here an example for the reconstruction of one layer of $X$ (the red one in Figure~\ref{fig:ex_algexp}) and the possible origin assignments on $B$. To construct this third layer of the dynamics of $X$, $\matrixt{B}{0,3}$ according to the origins of the nodes in $\tpmatrix{B}{0,1}$: $1$, $2$ and $-1$. Indeed, we have $\matrixt{B}{0,3|1}=\multiset{4,2,2,1,0,0}$ (the indegrees of $v_0$, $v_1$, $v_2$, $v_3$, $v_9$ and $v_{10}$), $\matrixt{B}{0,3|2}=\multiset{2,1,0}$ and $\matrixt{B}{0,3|-1}=\multiset{0,0,0}$. Then, we compute $M_1=\multiset{0}+\multiset{1}+\multiset{2}$ and $M_2=\multiset{0}\otimes(\multiset{2,1}+\multiset{3})=\multiset{0,0,0}$. 
After checking that $M_2$ is equal to $\matrixt{B}{0,3|-1}$, we can compute, with the $msDivision$ technique, $\matrixt{X}{0,3|1}$ such that $M_1\otimes\matrixt{X}{0,3|1}=\matrixt{B}{0,3|1}$ or $\multiset{2,1,0}\otimes\matrixt{X}{0,3|1}=\multiset{4,2,2,1,0,0}$. We found that $\matrixt{X}{0,3|1}=\multiset{1,2}$, which we can directly apply to the nodes in $X$ with label $3$ and $4$ (since they are connected to the node with origin $1$). 
We compute in the same way $\matrixt{X}{0,3|2}=\multiset{1}$. 

Finally, we assign to the nodes in $\tpmatrix{B}{0,2}$ the origins $3$, $4$ and $5$. Nodes $v_4$, $v_5$ and $v_{11}$ will receive the origin $5$ and nodes $v_6$, $v_7$ and $v_8$ will receive the origin $-1$. For the origins $3$ and $4$, we can fix one of each for the two nodes $v_9$ and $v_{10}$. In fact, the assignment chosen for nodes without predecessors will have no impact on the reconstruction of higher levels. For the nodes $(v_0,v_1,v_2,v_3)$, since both nodes in $X$ can produce a node with incoming degree $2$ in $B$, we will have to test the two assignments $(3,3,4,4)$ and $(3,4,3,4)$.

\subsubsection{Optimised assignment of origins.}

To reduce the number of possibilities and to avoid impossible assignments, we consider the height of the transients of $A$ and $B$. For some $\nodedds\in\setst$, let $\uheight(\nodedds)$ be the height of the sub-in-tree rooted at $\nodedds$ (with just the root that can be a cyclic node, in the case of $\nodedds\in\cycle$). 
For each node $v\in\tpmatrix{X}{\indexp,\tra-1|\ori}$, we know that it generates the subset $M_v=\multiset{d}\otimes M_1$. Let us consider an element $d'$ of $M_1$ and the corresponding node $u$ in $A$. 
According to our reasoning, one can give the label of $\nodedds$ to any node in $\tpmatrix{B}{\indexp,\tra-1|\ori}$ with $d \cdot d'$ predecessors. However, according to the product definition, a node $w$ of $B$, generated by $u$ of $A$ and $v$ of $X$, has $\uheight(w)=\min\set{\uheight(u),\uheight(v)}$. Then, we can only choose nodes $w$ in $\tpmatrix{B}{\indexp,\tra-1|\ori}$ with $d\cdot d'$ predecessors and such that $\uheight(w)\leq \uheight(u)$.
In line with the direct product definition, we can use this optimisation technique if both $u$ and $v$ are transient nodes.

\subsubsection{A concluding example}

Let us consider $A$, $X$, and $B$ of Figure \ref{fig:ex_algexp_height}.

\begin{figure}[ht!]
    \centering
    \begin{tikzpicture}
    \node[scale=1]() at (-5,-1){$A$};
    \node[scale=1]() at (-1.5,-1){$X$};
    \node[scale=1]() at (5,-1){$B$};

    \node[scale=0.9]() at (6.47,0.29){$w_{7}$};
    \node[scale=0.9]() at (6.20,0.90){$w_{6}$};
    \node[scale=0.9]() at (5.69,1.33){$w_{5}$};
    \node[scale=0.9]() at (5.04,1.50){$w_{4}$};
    \node[scale=0.9]() at (4.39,1.37){$w_{3}$};
    \node[scale=0.9]() at (3.85,0.97){$w_{2}$};
    \node[scale=0.9]() at (3.55,0.37){$w_{1}$};
    \node[scale=0.9]() at (3.53,-0.29){$w_{0}$};
    
    \node[scale=0.8]() at (5.65,-0.2){$(0)$};
    
    \node[scale=0.8]() at (-1.125,0.05){$(0)$};
    \node[scale=0.8]() at (-2.3,0.65){$(1)$};
    \node[scale=0.8]() at (-.7,0.65){$(2)$};
    
	\begin{oodgraph}
		\addcycle[xshift=-5.00cm, nodes prefix = a]{1};
		\addbeard[attach node = a->1,rotation angle=90]{2};
		\addbeard[attach node = a->1->1]{1};
		\addbeard[attach node = a->1->2]{2};
		\addbeard[attach node = a->1->2->1]{1};

		\addcycle[xshift=-1.5cm, nodes prefix = b]{1};
		\addbeard[attach node = b->1,rotation angle=90]{2};
		\addbeard[attach node = b->1->1]{1};
		\addbeard[attach node = b->1->2]{2};
		\addbeard[attach node = b->1->1->1,color=red]{1};

		\addcycle[xshift=5.00cm, nodes prefix = c]{1};
		\addbeard[attach node = c->1,rotation angle=90,opening angle=200,radius=1.5]{8};
		\addbeard[attach node = c->1->1,opening angle=35]{3};
		\addbeard[attach node = c->1->2,opening angle=55]{6};
		\addbeard[attach node = c->1->3,opening angle=35]{3};
		\addbeard[attach node = c->1->4,opening angle=35]{1};
		\addbeard[attach node = c->1->5,opening angle=35]{2};
		\addbeard[attach node = c->1->6,opening angle=55]{6};
		\addbeard[attach node = c->1->7,opening angle=35]{2};
		\addbeard[attach node = c->1->8,opening angle=40]{4};
		\addbeard[attach node = c->1->1->1,opening angle=25,rotation angle=-20]{3};
		\addbeard[attach node = c->1->1->2]{1};
		\addbeard[attach node = c->1->1->3,opening angle=25,rotation angle=20]{2};
		\addbeard[attach node = c->1->6->1,opening angle=25,rotation angle=-20]{3};
		\addbeard[attach node = c->1->6->2]{1};
		\addbeard[attach node = c->1->6->3,opening angle=25,rotation angle=20]{2};
		\addbeard[attach node = c->1->7->1]{1};
	\end{oodgraph}
    \end{tikzpicture}
    \caption{Relevance of height for the assignment of origins. The red part of $X$ is only shown to foretell the result of the reconstruction of the third level.}
    \label{fig:ex_algexp_height}
\end{figure}
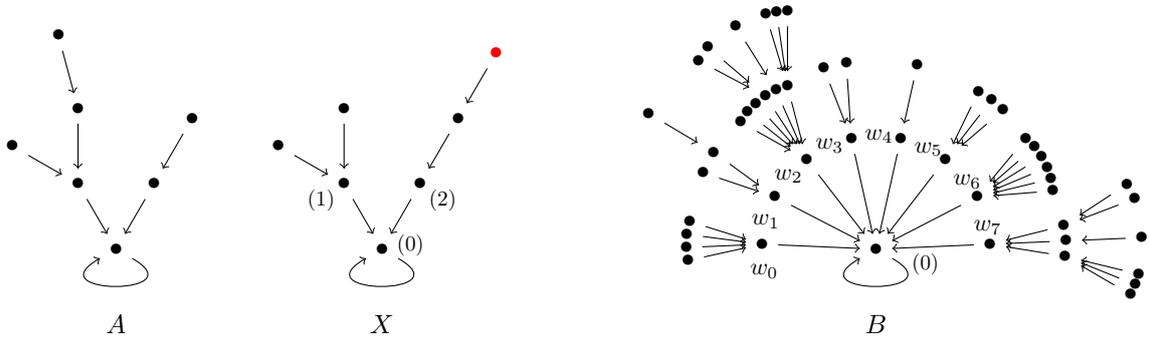

After having computed the second layer of $X$, one needs to assign one origin among three possible choices ($1$, $2$, and $-1$) for each of the eight nodes $w_i$ in $B$. Let us compute for each origin the multiset of degrees that the nodes must have:
\begin{align*}
    \ori=-1 & \rightarrow \matrixt{A}{0,2} \otimes \matrixt{X}{0,1} = \multiset{1,2}\otimes\multiset{3} = \multiset{3,6},\\
    \ori=1 & \rightarrow \multiset{2} \otimes (\matrixt{A}{0,2}+\matrixt{A}{0,1}) = \multiset{2}\otimes\multiset{1,2,3} = \multiset{2,4,6},\\
    \ori=2 & \rightarrow \multiset{1} \otimes (\matrixt{A}{0,2}+\matrixt{A}{0,1}) = \multiset{1}\otimes\multiset{1,2,3} = \multiset{1,2,3}.\\
\end{align*}

Following only the rules on the indegrees, some nodes can be given origins without ambiguity. The node $w_0$, being the only one having indegree $4$, must have origin $1$, and $w_4$ (with indegree $1$) must have origin $2$. However, several possibilities exist for the other ones:
\begin{itemize}
    \item $w_1$ and $w_3$ (with both indegree $2$) can either have the origin $1$ or $2$,
    \item $w_5$ and $w_7$ (with both indegree $3$) can either have the origin $-1$ or $2$,
    \item $w_2$ and $w_6$ (with both indegree $6$) can either have the origin $-1$ or $1$.
\end{itemize}
This gives us $8$ possible assignments. However, by looking at the heights of the subgraphs, we can reduce the number of possible assignments to only one. First, the nodes $w_i$ in $B$ originating from a node $u \in \tpmatrix{A}{0,1}$ and the cyclic point of $X$ (and hence having a $-1$ origin) must have $\uheight(w_i)=\uheight(u)$. As previously said, the height cannot be higher, and moreover, there must also be, in the in-tree rooted in $w_i$, a sub-in-tree corresponding to the in-tree rooted in $u$ (\ie, repeatedly multiplied with the cyclic point of $X$), which must have height $\uheight(u)$. As a consequence, $w_6$ and $w_7$ cannot be the nodes generated by the Cartesian product of the cyclic node of $X$ and the nodes in $\tpmatrix{X}{0,1}$ because $\uheight(w_6)<2$ and $\uheight(w_7)>1$.
On the other hand, the node having indegree $2$ and origin $1$ cannot have a rooted subtree of height exceeding $1$ since it is supposed to originate from a node $u' \in \tpmatrix{A}{0,1}$ with $\uheight(u')=1$. Therefore, $w_1$ can be excluded. Eventually, the origins for the nodes $w_0, \ldots, w_7$ are $1$, $2$, $-1$, $1$, $2$, $-1$, $1$ and $2$.

\subsubsection{Complexity.}
%
To compute the complexity of this approach, we must study the total number of assignments for a certain layer. To do so, one has to know the number of distinct $\matrixt{B}{\indexp,\tra|\ori}$ at this layer and their lengths. The first corresponds to the number of possible origins (\ie, the number of nodes in $X$ in a layer), which is $\frac{|\setst_X|}{\tra_B^{\max}}$ ($\tra_X^{\max}$ will be at most $\tra_B^{\max}$). For their sizes, the number of nodes in a layer of $B$ can be bounded as $\frac{|\setst_B|}{\tra_B^{\max}}$. 
However, the algorithm considers only the first $p_X$ values of $\indexp$, then we have $\frac{|\setst_B|\cdot p_X}{\tra_B^{\max}\cdot|\cycle_B|}$ as a bound.
From the explanation of the algorithm, the combinatorial assignment of origins is computed on the equal indegrees in a $\matrixt{B}{\indexp,\tra|\ori}$. We assume that the number of different incoming degrees is equal to the maximal one. We denote it as $d_{\max}$. Therefore, the total number $R$ of possible assignments of origins for a layer is $\left(\frac{|\setst_B|\cdot p_X}{|\setst_X|\cdot|\cycle_B|\cdot d_{\max}}!\right)^{\frac{|\setst_X|\cdot d_{\max}}{\tra_B^{\max}}}$. Since the origins are computed in every layer and this whole process is repeated for each alignment of $A$ and $B$, the total complexity is $O(|\cycle_A|\cdot(R^{\tra_B^{\max}}))$.

\subsubsection{An experimental evaluation.}
%
\begin{figure}
\begin{center}
\includegraphics[width=.9\textwidth]{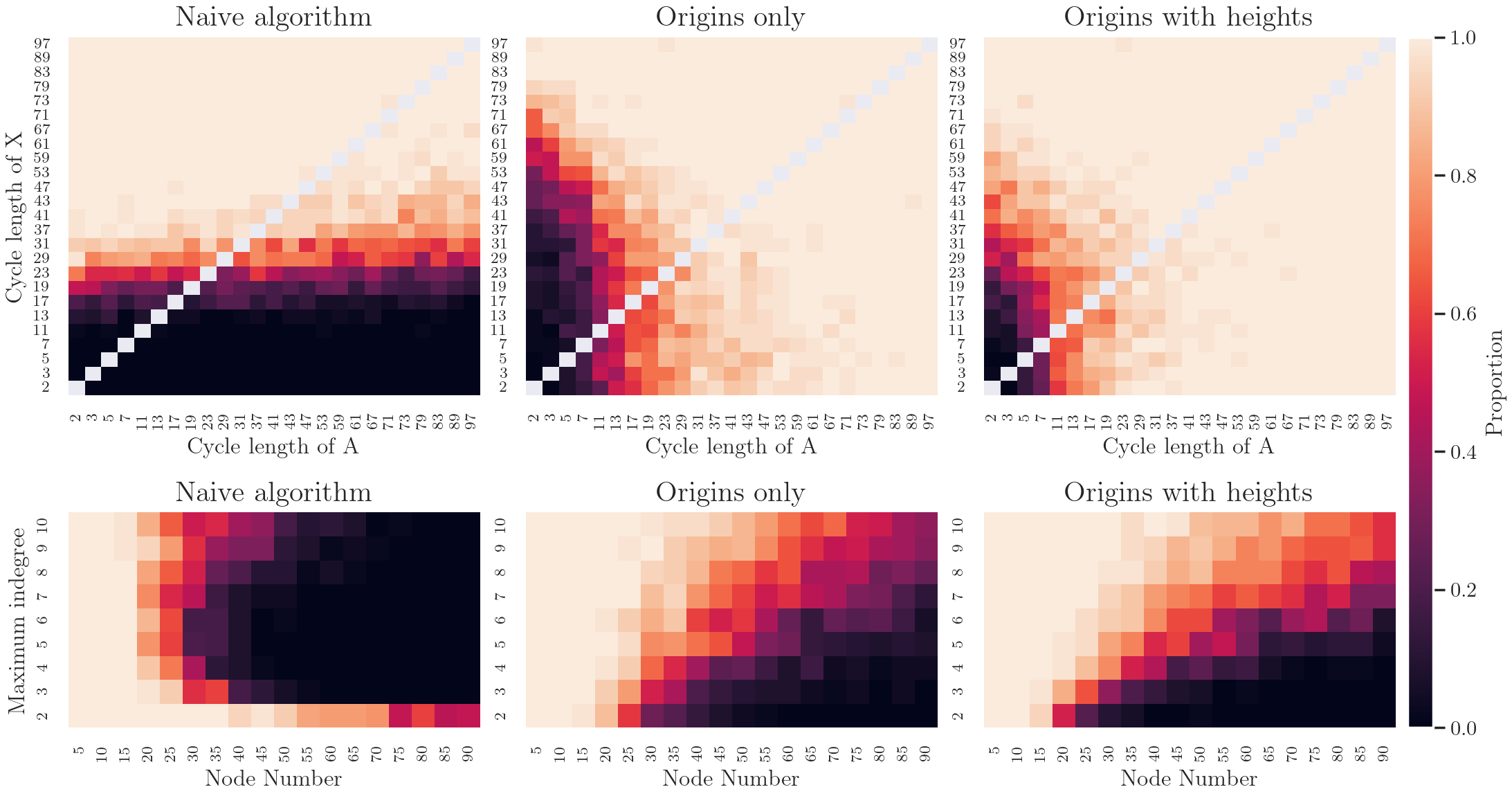}
\caption{Plots showing the proportion of random instances solved in less than $1$ second over the $50$ cases executed for each box. The three plots above are based on $30000$ instances with $A$ and $X$ of $100$ nodes each, and $B$ of $10000$ nodes, with increasing cycle lengths for $A$ ($x$-axis) or $X$ ($y$-axis). The three plots below are on $9500$ instances with $A$, $X$ and $B$ having only one cyclic node and an increasing number of nodes ($x$-axis) and maximum indegree ($y$-axis), for both $A$ and $X$.}
\label{fig:exp}
\end{center}
\end{figure}
%
%
We evaluated the exponential-time algorithm introduced in this section to investigate the performances over different instances\footnote{We ran all the experiments on an AMD EPYC 7301 processor running at 2.2 GHz, with 128 GiB of RAM. The algorithm has been implemented in Python 3.9.12 and run with version 7.3.9 of PyPy.}. The results of our analysis are shown in Figure \ref{fig:exp}. 
For this experimental evaluation, we generated random instances of equations over connected FGs according to different parameters. 

For the three plots above, we tested cases with $|\cycle_A|$ and $|\cycle_X|$ distinct prime numbers (this ensures $\gcd(|\cycle_A|,|\cycle_X|)=1$), since it represents the worst case scenario in which all nodes of $B$ are in one connected component.
As the total number of nodes is always the same for $A$ and $X$, we can remark that the axes also give the information on the number of transient nodes (inversely proportional to the cycle lengths).
We compared the naive reconstruction algorithm of $X$ (left column of Figure~\ref{fig:exp}) to our approach, with (right) and without (middle) height optimisation.
As one might have guessed, the naive algorithm strongly depends on the number of transient nodes of $X$. Indeed, the naive one does not take into account the information on the dynamics of $B$ and $A$, but just their t-abstractions to compute $\matrixt{X}{}$. Our approach, on the other hand, applies the combinatorial reasoning on elements of $B$ instead of $X$. It depends consequently on the number of transient nodes of $A$ and $X$, as $B$ does.
As explained in the section above, $\matrixt{B}{\indexp,\tra|\ori}$ comes from a multiplication of $M_1$ (which comes from $A$) and $\matrixt{X}{\indexp,\tra|\ori}$ (which comes from $X$). If one of the two is smaller, the resulting $\matrixt{B}{\indexp,\tra|\ori}$ will be smaller as well. This will decrease the number of possible assignments. The drawback is when $X$ has much fewer transient nodes than $A$. In this case, the combinatorics given by a simple naive reconstruction of the graph of $X$ from its abstraction would be more efficient as there are fewer possible cases. This explains why the naive algorithm performs better in the far left region of the plots. 

The optimisation over the assignment of origins turns out to be relevant to delaying the exponential growth of the algorithm. Note that, in emphasis with our conjecture, we considered the time up to the moment that the first solution is found. If we wanted to enumerate all solutions, the time would be multiplied by a linear factor.

As smaller cycle lengths tend to increase the execution time, we studied the extreme case of $|\cycle_A|$, $|\cycle_X|$ and $|\cycle_B|$ equal to $1$ (\ie, the plots below). As in the previous study, the naive algorithm depends almost exclusively on the number of transient points of $X$. It appears to be efficient only when the incoming degree of each node is at most $2$. It is a special case where many symmetries are involved and where a great number of distinct permutations lead to isomorphic dynamics. In the two other plots, the distribution of the incoming degree plays a greater role. One may think that greater indegrees lead to wider layers and to more combinatorics. Surprisingly, the maximum incoming degree decreases the execution time.
Actually, the combinatorial aspect is correlated to the number of nodes with the same indegree. This allows us to have much better results than the naive approach as the maximum indegree grows.
We notice that also in these cases the combinatorial optimisation improves the performances of the algorithm\footnote{The performance loss for low degrees comes from the graph generation routine which tends to create balanced in-trees. Thus, the height upgrade cannot overcome the extra time of the optimisation itself.}.

\section{Conclusion and Future Work}
\label{sec:conclusion}
We analysed basic linear equations over functional graphs. We proposed t-abstractions to model some information about the transient structure of FGs. By representing FGs mainly through the indegree of the transient nodes, we achieved three main results. First, we found an upper bound to the number of possible solutions to equations over t-abstractions, as well as for equations over FGs. Then, we introduced two different algorithms. The polynomial one (Algorithm~\ref{algpoly}) finds all solutions of basic equations over t-abstractions.
The second algorithm (Algorithm~\ref{algexpG}) reconstructs the actual graph solutions, which is not immediate even with the t-abstraction. We experimentally observed that this algorithm has good results in most cases and can compute solutions of equations with a large number of nodes in the known term, despite its exponential runtime. Finally, we introduced some further optimisations by considering interesting properties of the direct product computed over functional graphs and succeeded in pushing back a little more the exponential cost.

Future research directions comprise the further exploration of the connection with the cancellation problem. In particular, to prove Conjecture~\ref{conjecture} using our approach, and to better understand the correlation between the outgoing degree of digraphs and cancellation. Of course, another future development consists in trying to solve more complex equations. In particular, equations of the form $\matrixt{X}{} \cdot \matrixt{X}{} \supseteq \matrixt{B}{}$ represent an important step forward to the solution of generic polynomial equations over t-abstractions (with a constant right-hand side). We also want to set up a full pipeline to decide or find solutions to these more complex equations, by using the results on the basic equations, with a divide-and-conquer strategy.
Finally, it is necessary to investigate the form and the relations between different solutions (when they exist) of an equation over FGs, to discover if they exhibit some structure, for instance in terms of combinatorics, or from an algebraic point of view and to try to exploit these properties to conceive new improved algorithms.

\paragraph{Acknowledgements} Antonio E. Porreca was funded by his salary as a French State agent, affiliated to Aix-Marseille Université, CNRS, LIS, Marseille, France. This work was partially supported by the Agence National de la Recherche (ANR) project ANR-18-CE40-0002 FANs.

We also acknowledge the use of the TikZ library extension OOD used to make most of the figures in this paper.
The library can be found at \url{http://fgt.i3s.unice.fr}.

\bibliography{biblio}

\begin{thebibliography}{19}
\expandafter\ifx\csname natexlab\endcsname\relax\def\natexlab#1{#1}\fi
\providecommand{\url}[1]{\texttt{#1}}
\providecommand{\href}[2]{#2}
\providecommand{\path}[1]{#1}
\providecommand{\DOIprefix}{doi:}
\providecommand{\ArXivprefix}{arXiv:}
\providecommand{\URLprefix}{URL: }
\providecommand{\Pubmedprefix}{pmid:}
\providecommand{\doi}[1]{\href{http://dx.doi.org/#1}{\path{#1}}}
\providecommand{\Pubmed}[1]{\href{pmid:#1}{\path{#1}}}
\providecommand{\bibinfo}[2]{#2}
\ifx\xfnm\relax \def\xfnm[#1]{\unskip,\space#1}\fi
\bibitem[{Alonso{-}Sanz(2012)}]{AlonsoSanz12}
\bibinfo{author}{Alonso{-}Sanz, R.}, \bibinfo{year}{2012}.
\newblock \bibinfo{title}{Cellular automata and other discrete dynamical
  systems with memory}, in: \bibinfo{editor}{Smari, W.W.},
  \bibinfo{editor}{Zeljkovic, V.} (Eds.), \bibinfo{booktitle}{Proceedings of
  {HPCS}}, \bibinfo{publisher}{{IEEE}}. p. \bibinfo{pages}{215}.
\bibitem[{Bower and Bolouri(2004)}]{bower2004computational}
\bibinfo{author}{Bower, J.M.}, \bibinfo{author}{Bolouri, H.},
  \bibinfo{year}{2004}.
\newblock \bibinfo{title}{Computational modeling of genetic and biochemical
  networks}.
\newblock \bibinfo{publisher}{MIT press}.
\bibitem[{Dennunzio et~al.(2018)Dennunzio, Dorigatti, Formenti, Manzoni and
  Porreca}]{dorigatti2018}
\bibinfo{author}{Dennunzio, A.}, \bibinfo{author}{Dorigatti, V.},
  \bibinfo{author}{Formenti, E.}, \bibinfo{author}{Manzoni, L.},
  \bibinfo{author}{Porreca, A.E.}, \bibinfo{year}{2018}.
\newblock \bibinfo{title}{Polynomial equations over finite, discrete-time
  dynamical systems}, in: \bibinfo{editor}{Mauri, G.},
  \bibinfo{editor}{Yacoubi, S.E.}, \bibinfo{editor}{Dennunzio, A.},
  \bibinfo{editor}{Nishinari, K.}, \bibinfo{editor}{Manzoni, L.} (Eds.),
  \bibinfo{booktitle}{Cellular Automata - 13th International Conference on
  Cellular Automata for Research and Industry, {ACRI} 2018, Como, Italy,
  September 17-21, 2018, Proceedings}, \bibinfo{publisher}{Springer}. pp.
  \bibinfo{pages}{298--306}.
\bibitem[{Dennunzio et~al.(2019)Dennunzio, Formenti, Margara, Montmirail and
  Riva}]{dennunzio2019solving}
\bibinfo{author}{Dennunzio, A.}, \bibinfo{author}{Formenti, E.},
  \bibinfo{author}{Margara, L.}, \bibinfo{author}{Montmirail, V.},
  \bibinfo{author}{Riva, S.}, \bibinfo{year}{2019}.
\newblock \bibinfo{title}{Solving equations on discrete dynamical systems}, in:
  \bibinfo{editor}{Cazzaniga, P.}, \bibinfo{editor}{Besozzi, D.},
  \bibinfo{editor}{Merelli, I.}, \bibinfo{editor}{Manzoni, L.} (Eds.),
  \bibinfo{booktitle}{Computational Intelligence Methods for Bioinformatics and
  Biostatistics - 16th International Meeting, {CIBB} 2019, Bergamo, Italy,
  September 4-6, 2019, Revised Selected Papers}, \bibinfo{publisher}{Springer}.
  pp. \bibinfo{pages}{119--132}.
\bibitem[{Formenti et~al.(2021)Formenti, R{\'e}gin and Riva}]{formenti2021mdds}
\bibinfo{author}{Formenti, E.}, \bibinfo{author}{R{\'e}gin, J.C.},
  \bibinfo{author}{Riva, S.}, \bibinfo{year}{2021}.
\newblock \bibinfo{title}{Mdds boost equation solving on discrete dynamical
  systems}, in: \bibinfo{booktitle}{International Conference on Integration of
  Constraint Programming, Artificial Intelligence, and Operations Research},
  \bibinfo{organization}{Springer}. pp. \bibinfo{pages}{196--213}.
\bibitem[{Hammack and Smith(2014)}]{hammack2014zero}
\bibinfo{author}{Hammack, R.}, \bibinfo{author}{Smith, H.},
  \bibinfo{year}{2014}.
\newblock \bibinfo{title}{Zero divisors among digraphs}.
\newblock \bibinfo{journal}{Graphs and Combinatorics} \bibinfo{volume}{30},
  \bibinfo{pages}{171--181}.
\bibitem[{Hammack(2009)}]{hammack2009direct}
\bibinfo{author}{Hammack, R.H.}, \bibinfo{year}{2009}.
\newblock \bibinfo{title}{On direct product cancellation of graphs}.
\newblock \bibinfo{journal}{Discrete Mathematics} \bibinfo{volume}{309},
  \bibinfo{pages}{2538--2543}.
\bibitem[{Hammack et~al.(2011)Hammack, Imrich, Klav{\v{z}}ar, Imrich and
  Klav{\v{z}}ar}]{hammack2011handbook}
\bibinfo{author}{Hammack, R.H.}, \bibinfo{author}{Imrich, W.},
  \bibinfo{author}{Klav{\v{z}}ar, S.}, \bibinfo{author}{Imrich, W.},
  \bibinfo{author}{Klav{\v{z}}ar, S.}, \bibinfo{year}{2011}.
\newblock \bibinfo{title}{Handbook of product graphs}.
  volume~\bibinfo{volume}{2}.
\newblock \bibinfo{publisher}{CRC press Boca Raton}.
\bibitem[{Kadelka et~al.(2022)Kadelka, Laubenbacher, Murrugarra, Veliz-Cuba and
  Wheeler}]{kadelka2022decomposition}
\bibinfo{author}{Kadelka, C.}, \bibinfo{author}{Laubenbacher, R.},
  \bibinfo{author}{Murrugarra, D.}, \bibinfo{author}{Veliz-Cuba, A.},
  \bibinfo{author}{Wheeler, M.}, \bibinfo{year}{2022}.
\newblock \bibinfo{title}{Decomposition of boolean networks: An approach to
  modularity of biological systems}.
\newblock \bibinfo{journal}{arXiv preprint arXiv:2206.04217} .
\bibitem[{Katz(1955)}]{katz1955probability}
\bibinfo{author}{Katz, L.}, \bibinfo{year}{1955}.
\newblock \bibinfo{title}{Probability of indecomposability of a random mapping
  function}.
\newblock \bibinfo{journal}{The Annals of Mathematical Statistics} ,
  \bibinfo{pages}{512--517}.
\bibitem[{Kruskal(1954)}]{kruskal1954expected}
\bibinfo{author}{Kruskal, M.D.}, \bibinfo{year}{1954}.
\newblock \bibinfo{title}{The expected number of components under a random
  mapping function}.
\newblock \bibinfo{journal}{The American Mathematical Monthly}
  \bibinfo{volume}{61}, \bibinfo{pages}{392--397}.
\bibitem[{Lov{\'a}sz(1971)}]{lovasz1971cancellation}
\bibinfo{author}{Lov{\'a}sz, L.}, \bibinfo{year}{1971}.
\newblock \bibinfo{title}{On the cancellation law among finite relational
  structures}.
\newblock \bibinfo{journal}{Periodica Mathematica Hungarica}
  \bibinfo{volume}{1}, \bibinfo{pages}{145--156}.
\bibitem[{Metropolis and Ulam(1952)}]{metropolis1952property}
\bibinfo{author}{Metropolis, N.}, \bibinfo{author}{Ulam, S.},
  \bibinfo{year}{1952}.
\newblock \bibinfo{title}{A property of randomness of an arithmetical
  function}.
\newblock \bibinfo{type}{Technical Report}. Los Alamos Scientific Lab.
\bibitem[{Naquin and Gadouleau(2022)}]{naquin2022factorisation}
\bibinfo{author}{Naquin, {\'E}.}, \bibinfo{author}{Gadouleau, M.},
  \bibinfo{year}{2022}.
\newblock \bibinfo{title}{Factorisation in the semiring of finite dynamical
  systems}.
\newblock \bibinfo{journal}{arXiv preprint arXiv:2210.11270} .
\bibitem[{Romero and Zertuche(2003)}]{romero2003asymptotic}
\bibinfo{author}{Romero, D.}, \bibinfo{author}{Zertuche, F.},
  \bibinfo{year}{2003}.
\newblock \bibinfo{title}{The asymptotic number of attractors in the random map
  model}.
\newblock \bibinfo{journal}{Journal of Physics A: Mathematical and General}
  \bibinfo{volume}{36}, \bibinfo{pages}{3691}.
\bibitem[{Romero and Zertuche(2005)}]{romero2005grasping}
\bibinfo{author}{Romero, D.}, \bibinfo{author}{Zertuche, F.},
  \bibinfo{year}{2005}.
\newblock \bibinfo{title}{Grasping the connectivity of random functional
  graphs}.
\newblock \bibinfo{journal}{Studia Scientiarum Mathematicarum Hungarica}
  \bibinfo{volume}{42}, \bibinfo{pages}{1--19}.
\bibitem[{Schwab et~al.(2020)Schwab, Kühlwein, Ikonomi, Kühl and
  Kestler}]{SCHWAB2020571}
\bibinfo{author}{Schwab, J.D.}, \bibinfo{author}{Kühlwein, S.D.},
  \bibinfo{author}{Ikonomi, N.}, \bibinfo{author}{Kühl, M.},
  \bibinfo{author}{Kestler, H.A.}, \bibinfo{year}{2020}.
\newblock \bibinfo{title}{Concepts in boolean network modeling: What do they
  all mean?}
\newblock \bibinfo{journal}{Computational and Structural Biotechnology Journal}
  \bibinfo{volume}{18}, \bibinfo{pages}{571--582}.
\bibitem[{Sen{\'{e}}(2012)}]{Sene12}
\bibinfo{author}{Sen{\'{e}}, S.}, \bibinfo{year}{2012}.
\newblock \bibinfo{title}{On the bioinformatics of automata networks}.
\newblock Ph.D. thesis. University of \'{E}vry Val d'Essonne, France.
\bibitem[{Weichsel(1962)}]{weichsel1962kronecker}
\bibinfo{author}{Weichsel, P.M.}, \bibinfo{year}{1962}.
\newblock \bibinfo{title}{The {Kronecker} product of graphs}.
\newblock \bibinfo{journal}{Proceedings of the American {M}athematical
  {S}ociety} \bibinfo{volume}{13}, \bibinfo{pages}{47--52}.

\end{thebibliography}



\end{document}